\documentclass{article}
\usepackage{amsmath}
\usepackage{amsfonts}
\usepackage{amsthm}
\usepackage{amscd}
\usepackage{amssymb}
\usepackage{color}
\usepackage{tikz}

\newtheorem{theorem}{Theorem}[section]
\newtheorem{lemma}[theorem]{Lemma}
\newtheorem{proposition}[theorem]{Proposition}
\newtheorem{observation}[theorem]{Observation}
\newtheorem{problem}[theorem]{Problem}

\newtheorem{corollary}[theorem]{Corollary}

\newtheorem{note}[theorem]{Note}

\newtheorem{fact}[theorem]{Fact}

\theoremstyle{definition}
\newtheorem{definition}[theorem]{Definition}
\theoremstyle{remark}
\newtheorem{remark}[theorem]{Remark}

\newcommand{\1}{\mathbf{1}}

\newcommand{\are}{{\rm LF}\hskip0.02cm}

\newcommand{\card}{{\rm card}\hskip0.02cm}

\def\f2{\mathbb{F}_2}

\def\lip{\hskip0.02cm{\rm Lip}\hskip0.01cm}

\newcommand{\lf}{{\rm LF}\hskip0.02cm}

\newcommand{\ep}{\varepsilon}
\newcommand{\lin}{{\rm lin}\hskip0.02cm}

\newcommand{\diam}{{\rm diam}\hskip0.02cm}

\newcommand{\sign}{{\rm sign}\hskip0.02cm}

\begin{document}

\title{Lipschitz free spaces on finite metric spaces}

\author{Stephen J. Dilworth, Denka Kutzarova, and Mikhail I. Ostrovskii}

\date{\today}
\maketitle

\begin{large}

\begin{abstract}
Main results of the paper:

(1) For any finite metric space $M$ the Lipschitz free space on
$M$ contains a large well-complemented subspace which is close to
$\ell_1^n$.

(2) Lipschitz free spaces on large classes of recursively defined
sequences of graphs are not uniformly isomorphic to $\ell_1^n$ of
the corresponding dimensions. These classes contain well-known
families of diamond graphs and Laakso graphs.

Interesting features of our approach are: (a) We consider averages
over groups of cycle-preserving bijections of graphs which are not
necessarily graph automorphisms; (b) In the case of such recursive
families of graphs as Laakso graphs we use the well-known approach
of Gr\"unbaum (1960) and Rudin (1962) for estimating projection
constants in the case where invariant projections are not unique.
\end{abstract}



\tableofcontents

\section{Introduction}\label{S:Goals}

\subsection{Definitions and basic properties of Lipschitz free
spaces}

Basic facts about Lipschitz free spaces can be found in
\cite[Chapter 10]{Ost13} and \cite[Chapter 2]{Wea99} (in
\cite{Wea99} Lipschitz free spaces are called Arens-Eells spaces)

\begin{definition}\label{D:molLF} Let $X$ be a metric space. A
{\it molecule} of $X$\index{molecule of a metric space} is a
function $m:X\to\mathbb{R}$ which is supported on a finite set and
which satisfies $\sum_{p\in X} m(p)=0$. For $p,q\in X$ define the
molecule $m_{pq}$ by $m_{pq}=\1_p-\1_q$, where $\1_p$ and $\1_q$
are indicator functions of singleton sets $\{p\}$ and $\{q\}$. We
endow the space of molecules with the seminorm
\[||m||_{\are}=\inf\left\{\sum_{i=1}^n|a_i|d_X(p_i,q_i):~m=\sum_{i=1}^na_im_{p_iq_i}\right\}\]
It is not difficult to see that this is actually a norm. The {\it
Lipschitz free space} over $X$\index{Lipschitz free space over a
metric space} is defined as the completion of the space of all
molecules with respect to the norm $||\cdot||_{\are}$. We denote
the Lipschitz free space over $X$ by $\are(X)$.
\end{definition}

By a {\it pointed metric space} we mean a metric space with a
distinguished point, denoted $O$. By $\lip_0(X)$ we denote the
space of all Lipschitz functions $f:X\to\mathbb{R}$ satisfying
$f(O)=0$, where $O$ is the distinguished point of a pointed metric
space $X$. It is not difficult to check that $\lip_0(X)$ is a
Banach space with respect to the norm $||f||=\lip(f)$. As is
well-known \cite{Ost13,Wea99}, the following duality holds:
\begin{equation}\label{E:LFdual}\are(X)^*=\lip_0(X).\end{equation}

We also need the following description of $\lf(X)$ in the case
where $X$ is a vertex set of an unweighted graph with its graph
distance. Let $G=(V(G),E(G))=(V,E)$ be a finite graph. Let
$\ell_1(E)$ be the space of real-valued functions on $E$ with the
norm $||f||=\sum_{e\in E}|f(e)|$. We consider some orientation on
$E$, so each edge of $E$ is a directed edge. For a directed cycle
$C$ in $E$ (we mean that the cycle can be ``walked around''
following the direction, which is not related with the orientation
of $E$) we introduce the {\it signed indicator function} of $C$ by
\begin{equation}\label{E:SignInd}
\chi_C(e)=\begin{cases} 1 & \hbox{ if }e\in C\hbox{ and its
orientations in $C$ and $G$ are the same}\\
-1 & \hbox{ if }e\in C\hbox{ but its orientations in $C$ and $G$
are different}
\\
0 & \hbox{ if }e\notin C.
\end{cases}\end{equation}

The {\it cycle space} $Z(G)$ of $G$ is the subspace of $\ell_1(E)$
spanned by the signed indicator functions of all cycles in $G$. We
will use the fact that $\lf(G)$ for unweighed graphs $G$
(\cite[Proposition 10.10]{Ost13}) is isometrically isomorphic to
the quotient of $\ell_1(E)$ over $Z(G)$:
\begin{equation}\label{E:LFunweigh}\lf(G)=\ell_1(E)/Z(G)\end{equation}

We use the standard terminology of Banach space theory
\cite{BL00}, graph theory \cite{BM08,Die17}, and the theory of
metric embeddings \cite{Ost13}.

\subsection{Historical and terminological remarks}

The Lipschitz-free spaces are studied by several groups of
researchers, for different reasons and under different names. Some
authors use the term {\it Arens-Eells space} (see
\cite{Kal08,Wea99}), which reflects the contribution of Arens and
Eells \cite{AE56}. The norm of this space and a more general space
of measures (see \cite{Vil03,Vil09,Wea99}) is called the {\it
Kantorovich-Rubinstein distance (or norm)}  to acknowledge the
contribution of Kantorovich and Rubinstein \cite{Kan42,KR58}, or
{\it Wasserstein distance (or norm)}, (see \cite{ANN17+, NR17}) to
acknowledge the contribution of Wasserstein \cite{Vas69} (whose
name is transliterated from Russian as Vasershtein), see the paper
\cite{Dob70}, where the term Wasserstein distance was introduced.
The term {\it Wasserstein norm} is also used (and more
justifiably) for the $p$-analogue of the distance. The term {\it
Lipschitz free space} is commonly used (especially in the Banach
space theory) after the publication of the paper \cite{GK03}. The
names used for this distance in Computer Science are {\it earth
mover distance} and {\it transportation cost} (see \cite{ADIW09},
\cite{AIK08}, \cite{KN06}, \cite{NS07}). All of the mentioned
above notions are equivalent for finite metric spaces which we
consider in this paper. For this reason we decided not to attach
any of the mentioned names to the objects of our study and to use
the neutral name {\it Lipschitz free space} (which only reflects
the connection of this notion with the notion of a Lipschitz
function).

Lipschitz free spaces are of significant interest for Computer
Science (see \cite{IM04}), Functional Analysis (\cite{God15}, \cite{Kal08},
\cite{Wea99}), Metric Geometry (\cite{ANN17+},
\cite[p.~134]{NR17}, \cite{Ost13}), Optimal Transportation
(\cite{Vil03},\cite{Vil09}).

\subsection{Overview of the paper}

Our interest in Lipschitz free spaces is inspired by the theory of
Metric Embeddings (see \cite{Ost13}): we are interested in
studying properties of Banach spaces admitting an isometric
embedding of a given metric space. We are going to focus on finite
metric spaces.
\medskip

Our main results and observations:

\begin{enumerate}

\item[{\bf 1.}] We show that for any finite metric space $M$ the
space $\lf(M)$ contains a half-dimensional well-complemented
subspace which is close to $\ell_1^n$, see Section
\ref{S:LargeL1}.

\item[{\bf 2.}] We prove that the Lipschitz free spaces on large
classes of recursively defined sequences of graphs (see Section
\ref{S:DefRecDiLa} for definitions) are not uniformly isomorphic
to $\ell_1^n$ of the corresponding dimensions (Section
\ref{S:Recur}). These classes contain well-known families of
diamond graphs and Laakso graphs, see \ref{S:DefRecDiLa} for
definitions and Section \ref{S:Diamond} for proofs. The case of
diamond graphs can also be handled using classical theory of
orthogonal series. Since this approach has its advantages and
leads to more precise results, we enclose the corresponding
argument in Section \ref{S:DiamHaar}.

Interesting features of our approach are: (1) We consider averages
over groups of cycle-preserving bijections of graphs which are not
necessarily automorphisms (see Section \ref{S:CycBij}); (2) In the
case of such recursive families of graphs as Laakso graphs we use
the well-known approach of Gr\"unbaum \cite{Gru60} and Rudin
\cite{Rud62} for estimating projection constants in the case where
invariant projections are not unique (see Sections
\ref{S:GRAAppr}, \ref{S:Annihil}, \ref{S:Combin}, and
\ref{S:NonUnLaakso}).

\item[{\bf 3.}] We observe (Section \ref{S:L1}) that the known
fact (see \cite{Dal15}, \cite{CD16}) that Lipschitz free spaces on
finite ultrametrics are close to $\ell_1$ in the Banach-Mazur
distance immediately follows from the result of Gupta \cite{Gup01}
on Steiner points and the well-known result on isometric
embeddability of ultrametrics into weighted trees.

\item[{\bf 4.}] We finish this section by observing that the
result of Erd\H{o}s and P\'osa \cite{EP62} on edge-disjoint cycles
implies that the cycle space (considered as a subspace of
$\ell_1(E)$) always contains a `large' $1$-complemented in
$\ell_1(E)$ subspace isometric to $\ell_1^n$.

Observe that the subspace in $Z(G)$ spanned by the signed
indicator functions of a family of edge-disjoint cycles is
isometric to $\ell_1^n$ of the corresponding dimension and is
$1$-complemented in $\ell_1(E(G))$, and so in $Z(G)$.

This makes us interested in the estimates of the amount of
edge-disjoint cycles in terms of the dimension of the cycle space.
Such estimates, sharp up to the constants involved in them, were
obtained by Erd\H{o}s and P\'osa \cite{EP62}. Denote by $\mu(G)$
the dimension of the cycle space of $G$. It is well-known, see
\cite[Proposition 2.1]{Big97}, that for connected graphs
$\mu(G)=|E(G)|-|V(G)|+1$. Let $\nu(G)$ be the maximal number of
edge-disjoint cycles in $G$.

\begin{theorem}[{Erd\H{o}s and P\'osa \cite[Theorem
4]{EP62}}]\label{T:EP}
\[\nu(G)=\Omega\left(\frac{\mu(G)}{\log(\mu(G))}\right)\]
and for some family of graphs
\[\nu(G)=O\left(\frac{\mu(G)}{\log(\mu(G))}\right)\]
\end{theorem}

\begin{remark} It is worth mentioning that Erd\H{o}s and P\'osa
state their result slightly differently. They do not require
graphs to be simple or connected and denote by $g(k)$ the smallest
integer such that for any $n\in \mathbb{N}$ a graph with $n$
vertices and $n+g(k)$ edges contains at least $k$ edge-disjoint
cycles. Theorem 4 in \cite{EP62} states that
\[g(k)=\Theta(k\log k).\]
It is easy to see that Theorem \ref{T:EP} follows from this
result.
\end{remark}

\end{enumerate}

\subsection{Recursive families of graphs, diamond graphs and Laakso
graphs}\label{S:DefRecDiLa}

We are going to use the general definition of recursive sequences
of graphs introduced by Lee and Raghavendra \cite{LR10}.

\begin{definition}\label{D:Comp}
Let $H$ and $G$ be two finite connected directed graphs having
distinguished vertices which we call {\it top} and {\it bottom},
respectively. The {\it composition} $H\oslash G$ is obtained by
replacing each  edge $\overrightarrow{uv}\in E(H)$ by a copy of
$G$, the vertex $u$ is identified with the bottom of $G$ and the
vertex $v$ is identified with the top of $G$. Directions of edges
in $H\oslash G$ are inherited from $G$. The {\it top} and {\it
bottom} of the obtained graph are defined as the top and bottom of
$H$, respectively.
\end{definition}

When we consider these graphs as metric spaces we use the graph
distances of the underlying undirected graphs (that is, we ignore
the directions of edges).\medskip

The following property of this composition is straightforward to
verify:

\begin{lemma}[Associativity of $\oslash$]\label{L:Assoc}  For any three graphs
$F,G,H$ the sides of
\begin{equation}\label{E:Assoc}
(F\oslash G)\oslash H=F\oslash (G\oslash H),
\end{equation}
are equal both as directed graphs and as metric spaces.
\end{lemma}

Let $B$ be a connected unweighted finite simple directed graph
having two distinguished vertices, which we call {\it top} and
{\it bottom}, respectively. We use $B$ to construct recursive
family of graphs as follows:

\begin{definition}\label{D:B_n} We say that the graphs $\{B_n\}_{n=0}^\infty$ are defined by {\it recursive
composition} or that $\{B_n\}_{n=0}^\infty$ is a {\it recursive
sequence} or {\it recursive family} of graphs if:

\begin{itemize}

\item The graph $B_0$ consists of one directed edge with {\it
bottom} being the initial vertex and {\it top} being the terminal
vertex.

\item $B_n=B_{n-1}\oslash B$.

\end{itemize}
\end{definition}

Observe that Lemma \ref{L:Assoc} implies that for every
$k\in\{0,1,\dots,n\}$ we have
\begin{equation}\label{E:DiffProd} B_n=B_{n-k}\oslash B_{k},
\end{equation} also $B_1=B$. The authors of \cite{LR10} use the notation
$B_n=B^{\oslash n}$.\medskip

Observe that in the case where the graph $B$ has an automorphism
which maps its bottom to top and top to bottom, the choice of
directions on edges will not affect the isomorphic structure of
the underlying undirected  graphs. For this reason to define
recursive families in such cases we do not not need to assign
directions to edges.
\medskip

Interesting and important examples of recursive families of graphs
have been extensively studied in the literature. One of the most
well-known and important for the theory of metric embeddings
families was introduced in \cite{GNRS04} (conference version was
published in 1999). This family (which turned to be very useful in
the theory of metric characterizations of classes of Banach spaces
\cite{JS09}, see also \cite[Section 9.3.2]{Ost13}) corresponds to
the special case of Definition \ref{D:B_n}, where $B$ is a square
and one pair of its opposite vertices is chosen to play roles of
the top and the bottom. The usual definition of diamond graphs is
the following.

\begin{definition}[Diamond graphs]\label{D:Diamonds}
Diamond graphs $\{D_n\}_{n=0}^\infty$ are defined recursi\-ve\-ly:
The {\it diamond graph} of level $0$ has two vertices joined by an
edge of length $1$ and is denoted by $D_0$. The {\it diamond
graph} $D_n$ is obtained from $D_{n-1}$ in the following way.
Given an edge $uv\in E(D_{n-1})$, it is replaced by a
quadrilateral $u, a, v, b$, with edges $ua$, $av$, $vb$, $bu$.
(See Figure \ref{F:Diamond2}.)
\end{definition}

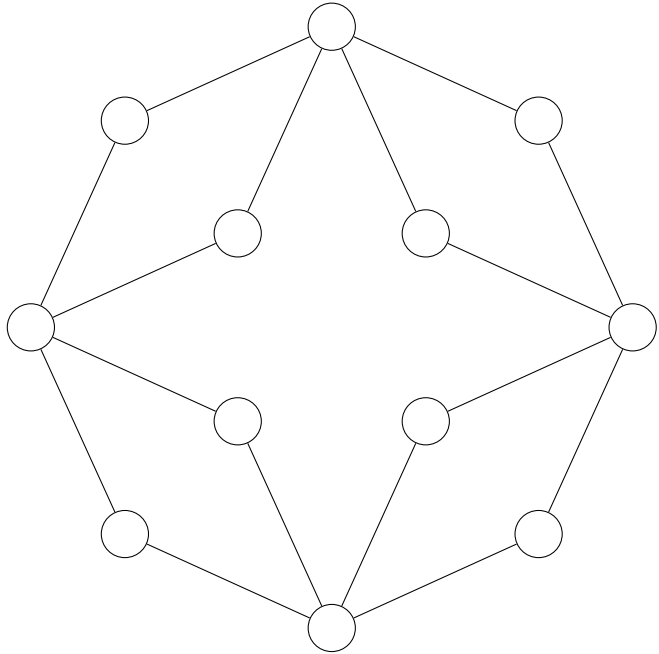
\begin{figure}
\begin{center}
{
\begin{tikzpicture}
  [scale=.25,auto=left,every node/.style={circle,draw}]
  \node (n1) at (16,0) {\hbox{~~~}};
  \node (n2) at (5,5)  {\hbox{~~~}};
  \node (n3) at (11,11)  {\hbox{~~~}};
  \node (n4) at (0,16) {\hbox{~~~}};
  \node (n5) at (5,27)  {\hbox{~~~}};
  \node (n6) at (11,21)  {\hbox{~~~}};
  \node (n7) at (16,32) {\hbox{~~~}};
  \node (n8) at (21,21)  {\hbox{~~~}};
  \node (n9) at (27,27)  {\hbox{~~~}};
  \node (n10) at (32,16) {\hbox{~~~}};
  \node (n11) at (21,11)  {\hbox{~~~}};
  \node (n12) at (27,5)  {\hbox{~~~}};

  \foreach \from/\to in {n1/n2,n1/n3,n2/n4,n3/n4,n4/n5,n4/n6,n6/n7,n5/n7,n7/n8,n7/n9,n8/n10,n9/n10,n10/n11,n10/n12,n11/n1,n12/n1}
    \draw (\from) -- (\to);

\end{tikzpicture}
} \caption{Diamond $D_2$.}\label{F:Diamond2}
\end{center}
\end{figure}

Let us count some parameters associated with graphs $D_n$. Denote
by $V(D_n)$ and $E(D_n)$ the vertex set and edge set of $D_n$,
respectively. We need the following simple observations about
cardinalities of these sets:

\begin{enumerate}

\item[{\bf (a)}] $|E(D_n)|=4^n$.

\item[{\bf (b)}] $|V(D_{n+1})|=|V(D_n)|+2|E(D_n)|$.\label{P:(b)}

\end{enumerate}

Hence $|V(D_n)|=2(1+\sum_{i=0}^{n-1}4^i)$.
\medskip

The next special case of the general Definition \ref{D:B_n}, whose
metric geometry was studied in \cite{LR10,OR17}, corresponds to
the case where $B=K_{2,n}$, and the vertices in the part
containing $2$ vertices play the roles of the top and the bottom.
The usual definition is the following.

\begin{definition}[Multibranching diamonds]\label{D:BranchDiam}
For  any integer  $k\ge 2$, we define $D_{0,k}$ to be the graph
consisting of two vertices joined by one edge. For any
$n\in\mathbb{N}$, if the graph $D_{n-1,k}$  is already defined,
the graph $D_{n,k}$ is defined as the graph obtained from
$D_{n-1,k}$ by replacing each edge $uv$ in $D_{n-1,k}$ by a set of
$k$ independent paths of length $2$ joining $u$ and $v$. We endow
$D_{n,k}$ with the shortest path distance. We call
$\{D_{n,k}\}_{n=0}^\infty$ {\it diamond graphs of branching $k$},
or {\it diamonds of branching} $k$.
\end{definition}

The last special case of the general Definition \ref{D:B_n}, which
we consider in this paper goes back to the paper of Laakso
\cite{Laa00}. The corresponding recursive family of graphs was
introduced by Lang and Plaut \cite{LP01}. In \cite{OO17} it was
shown that these graphs are incomparable with diamond graphs in
the following sense: elements of none of these families admit
bilipschitz embeddings into the other family with uniformly
bounded distortions. Laakso graphs correspond to the case where
the graph $B$ is the graph shown in Figure \ref{F:Laakso} with the
natural choice for the top and the bottom.

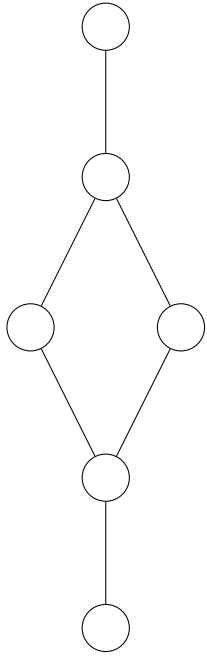
\begin{figure}
\begin{center}
{
\begin{tikzpicture}
  [scale=.25,auto=left,every node/.style={circle,draw}]
  \node (n1) at (16,0) {\hbox{~~~}};
  \node (n2) at (16,8)  {\hbox{~~~}};
  \node (n3) at (12,16) {\hbox{~~~}};
  \node (n4) at (20,16)  {\hbox{~~~}};
  \node (n5) at (16,24)  {\hbox{~~~}};
  \node (n6) at (16,32) {\hbox{~~~}};

\foreach \from/\to in {n1/n2,n2/n3,n2/n4,n4/n5,n3/n5,n5/n6}
    \draw (\from) -- (\to);

\end{tikzpicture}
} \caption{Laakso graph $\mathcal{L}_1$.}\label{F:Laakso}
\end{center}
\end{figure}

\begin{definition}\label{D:Laakso}
Laakso graphs $\{\mathcal{L}_n\}_{n=0}^\infty$ are defined
recursively: The {\it Laakso graph} of level $0$ has two vertices
joined by an edge of length $1$ and is denoted $\mathcal{L}_0$.
The {\it Laakso graph} $\mathcal{L}_n$ is obtained from
$\mathcal{L}_{n-1}$ according to the following procedure. Each
edge $uv\in E(\mathcal{L}_{n-1})$ is replaced by the graph
$\mathcal{L}_1$ exhibited in Figure \ref{F:Laakso}, the vertices
$u$ and $v$ are identified with the vertices of degree $1$ of
$\mathcal{L}_1$.
\end{definition}

\section{Lipschitz-free spaces close to $\ell_1^n$}\label{S:L1}

Our first proposition is known, see \cite[Corollary 3.3]{God10},
we give a direct proof of it for convenience of the reader.

\begin{proposition}\label{P:Tree} Let $T$ be a finite weighted tree. Then  $\lf(T)$ is
isometric to $\ell_1^k$, where $k$ is the number of edges in the
tree.
\end{proposition}

\begin{proof} Let $f\mapsto e_f$ be a bijection between the edge set of $T$ and the
unit vector basis in $\ell_1^k$. We denote the weight of $f$ by
$w(f)$. We consider the following map $F$ of the set of molecules
on $T$ into $\ell_1^k$.

For each edge $f=\{u,v\}$ we let $F(\1_u-\1_v)=w(f)e_f$. It is
clear that each molecule in $\lf(T)$ can be (uniquely) written as
a linear combination of molecules $\{\1_u-\1_v\}_{\{u,v\}\in
E(T)}$. We define $F$ to be the linear extension of the defined
map to $\lf(T)$, it is clear from this definition that $F$ is a
surjective map onto $\ell_1^k$.

By the duality \eqref{E:LFdual}, to show that $F$ is an isometry
of $\lf(T)$ onto $\ell_1^k$ it is enough to find a $1$-Lipschitz
function $L\in\lip_0(T)$ (the base point $O$ is chosen
arbitrarily) such that
\[L\left(\sum_{\{u,v\}\in E(T)}a_{uv}(\1_u-\1_v)\right)=\sum_{\{u,v\}\in E(T)}|a_{uv}| \cdot
w(uv),\] where $a_{uv}\in\mathbb{R}$.

Construction of such $1$-Lipschitz function $L$ is quite
straightforward. We let $L(O)=0$. If the function is already
defined on one end $u$ of an edge $\{u,v\}$, we set $L(v)=L(u)\pm
w(uv)$, where we choose $+$ if the coefficient of $\1_v-\1_u$ in
$m$ is nonnegative, and $-$ if the coefficient of $\1_v-\1_u$ in
$m$ is negative. It is clear that $L$ is $1$-Lipschitz and
$L(m)=\sum_{\{u,v\}\in E(T)}|a_{uv}| \cdot w(uv)$.
\end{proof}

The following result is very useful in the current context.

\begin{theorem}[\cite{Gup01}]\label{T:Gupta} Let $T$ be a weighted tree and $M$ be a subset of $V(T)$. Then there is a weighted
tree $\tilde T$ with the vertex set $M$ such that the distances
induced by $T$ and $\tilde T$ on $M$ are $8$-equivalent.
\end{theorem}

\begin{corollary}\label{C:SubsL1} Let $T$ be a weighted tree and $M$ be a subset of $V(T)$.
Then the Banach-Mazur distance between $\lf(M)$ (where $M$ is
endowed with the metric induced from $T$) and $\ell_1^k$ of the
corresponding dimension does not exceed $8$.
\end{corollary}

\begin{remark} Gupta \cite{Gup01} did not show that the constant
$8$ is sharp, his lower estimate for the constant is $4$. It is
not clear what is the optimal constant in Corollary
\ref{C:SubsL1}.
\end{remark}

Since it is well-known that ultrametrics can be isometrically
embedded into weighted trees (see, for example, \cite[Theorem
9]{CM10}, and also \cite[Section 3]{DSW17}), we get also the
following finite version of results of \cite{Dal15} and
\cite{CD16}:

\begin{corollary}\label{C:Ultra} Let $M$ be a finite ultrametric space. Then  $\lf(M)$ is
$8$-isomorphic to $\ell_1^k$.
\end{corollary}

To see that there are metric spaces of different nature whose
Lipschitz free spaces are also close to $\ell_1^k$ of the
corresponding dimension, we use \eqref{E:LFunweigh}. This equality
implies that if we consider a graph $G$ which contains small
amount of cycles, or all cycles in it are disjoint, then $\lf(G)$
is close to $\ell_1^n$ of the corresponding dimension.

The space $\lf(G)$ remains close to $\ell_1^n$ for metric spaces
which are bilipschitz equivalent to graphs having properties
described in the previous paragraph. One of the ways of getting
such metric spaces is deletion of edges forming short cycles, see
\cite{Pel00} on results related to this construction, especially
\cite[Section 17.2]{Pel00}. It is worth mentioning that
bilipschitz equivalent metric spaces can have quite different
structure of cycle spaces. Consider, for example, $K_n$ (complete
graph on $n$ vertices) and the graph $K_{1,n-1}$ consisting of $n$
vertices in which the first vertex is adjacent to all other
vertices, and there are no other edges. Any bijection between
these metric spaces has distortion $2$, the cycle space $Z(K_n)$
is a large space whereas $Z(K_{1,n-1})$ is trivial.

\begin{problem}\label{P:CloseToell1} It would be very interesting to find a condition on a finite metric space $M$ which is equivalent
to the condition that the space $\lf(M)$ is Banach-Mazur close to
$\ell_1^n$ of the corresponding dimension. It is not clear whether
it is feasible to find such a condition.
\end{problem}

\section{Large complemented $\ell_1^n$ in finite-dimensional
Lipschitz\\ free spaces}\label{S:LargeL1}

The following result can be regarded as a finite-dimensional
version of the result of C\'uth, Doucha, and Wojtaszczyk
\cite{CDW16} who proved that the Lipschitz free space on an
infinite metric space contains a complemented subspace isomorphic
to $\ell_1$.

\begin{theorem}\label{T:HalfDimL1} For every $n$-point metric space $M$ the
space $\lf(M)$ contains a $2$-complemented $2$-isomorphic copy of
$\ell_1^k$ with $k=\lfloor\frac{n}2\rfloor$.
\end{theorem}

The following lemma is a version of \cite[Lemma 3.1]{CDW16}.

\begin{lemma}\label{L:CDW} Let $(M,d)$ be a finite metric space and $\{y_i\}_{i=1}^k$ be a sequence of distinct points in
$M$ such that $M\backslash \{y_i\}_{i=1}^k$ is nonempty. For each
$i\in\{1,\dots,k\}$ let $x_i\in M\backslash \{y_i\}_{i=1}^k$ be
such that the distance $d(x_i,y_i)$ is minimized, so
$\{x_i\}_{i=1}^k$ are not necessarily distinct. Then linear
combinations of the functions $f_i(x)=d(y_i,x_i)\1_{y_i}(x)$
satisfy the inequality
\begin{equation}\label{E:LinfLip}
\max_i|\alpha_i|\le
\lip\left(\sum_{i=1}^k\alpha_if_i\right)\le\max\left\{\max_{i\ne
j}\frac{d(x_i,y_i)+d(x_j,y_j)}{d(y_i,y_j)}\,,1\right\}\cdot\max_i|\alpha_i|
\end{equation}
\end{lemma}

\begin{proof} The leftmost inequality is obtained by comparing
the values of $\sum_{i=1}^k\alpha_if_i$ at $x_m$ and $y_m$, where
$m\in\{1,\dots,k\}$ is such that $\alpha_m=\max_i|\alpha_i|$.

To prove the rightmost inequality we do the following analysis:
consider any pair $(u,v)$ of points in $M$ and estimate from above
the quotient
\[\left|\sum_{i=1}^k\alpha_if_i(u)-\sum_{i=1}^k\alpha_if_i(v)\right|/d(u,v).\]

If the points $u$ and $v$ are $y_i$ and $y_j$, $i\ne j$, then the
estimate from above is
\[\displaystyle{\frac{d(x_i,y_i)+d(x_j,y_j)}{d(y_i,y_j)}\cdot\max_i|\alpha_i|}.\]
If one of the points is $y_i$ and the other is not in the sequence
$\{y_i\}_{i=1}^k$, we get at most $\max_i|\alpha_i|$, because of
the minimality property of $d(x_i,y_i)$. If both $u$ and $v$ are
not in $\{y_i\}_{i=1}^k$, then
$\sum_{i=1}^k\alpha_if_i(u)=\sum_{i=1}^k\alpha_if_i(v)=0$.
\end{proof}

\begin{proof}[Proof of Theorem \ref{T:HalfDimL1}]
Any finite metric space can be considered as a weighted graph with
the weighted graph distance (we may consider elements of the
metric space as vertices of a complete graph with the weight of
each edge equal to the distance between its ends).

Consider the minimum weight spanning tree $T$ in this graph
constructed according to Boruvka-Kruskal procedure
\cite[Construction A]{Kru56} (see also \cite[Algorithm 8.22]{BM08}), that is: we list edges in the order
of nondecreasing lengths; then we process this list from the
beginning and pick for the spanning tree all edges which do not
form cycles with the previously selected.

It is easy to see that the picked set of edges satisfies the
following condition: at least one of the shortest edges incident
to each of the vertices is in the spanning tree.

Any tree is a bipartite graph. Therefore we can split $M$ into two
subsets, $M_1$ and $M_2$, such that any edge in the spanning tree
$T$ has one vertex in $M_1$ and the other in $M_2$. At least one
of the sets $M_1$ and $M_2$ contains at least half of elements of
$M$. We assume that $M_1$ is such and label its vertices as
$\{y_i\}_{i=1}^k$. For each $i$ we let $x_i$ be the closest to
$y_i$ element of $M_2$ (the elements $\{x_i\}_{i=1}^k$ are not
required to be distinct). The comment in the previous paragraph
implies that $x_i\in M_2$ is one of the closest to $y_i$ and
different from $y_i$ elements of $M$. Hence, by Lemma \ref{L:CDW},
the subspace of $\lip_0(M)$ (we pick the base point to be any
element of $M\backslash \{y_i\}_{i=1}^k$) spanned by
$\{d(y_i,x_i)\1_{y_i}\}_{i=1}^k$ is $2$-isomorphic to
$\ell_\infty^k$ and thus $2$-complemented in $\lip_0(M)$.

Consider the functions $u_i=(\1_{y_i}-\1_{x_i})/d(x_i,y_i)$ in
$\lf(M)$. We claim that $\{u_i\}_{i=1}^k$ span a $2$-complemented
subspace in $\lf(M)$ which is $2$-isomorphic to $\ell_1^k$. It is
clear that $||u_i||=1$ and $f_i(u_j)=\delta_{i,j}$ (Kronecker
$\delta$). Let $\{b_i\}_{i=1}^k$ be a sequence of real numbers
satisfying $\sum_{i=1}^k|b_i|=1$, and $x=\sum_{i=1}^k b_iu_i$. We
need to estimate the norm of $x$. Clearly
$||x||\le\sum_{i=1}^k|b_i|=1$.

On the other hand, let $\alpha_i=\sign(b_i)$. Then, by the first
part of the proof,
\[1\le\left\|\sum_{i=1}^k\alpha_if_i\right\|\le 2\]
On the other hand
\[\left(\sum_{i=1}^k\alpha_if_i\right)(x)=\sum_{i=1}^k\alpha_i
b_i=\sum_{i=1}^k |b_i|.\] Hence $\frac12\le||x||\le 1$.

Now we show that the linear span of $\{u_i\}$ is complemented. We
introduce $P:\lf(M)\to\lin\{u_i\}$ by
\[P(u)=\sum_{i=1}^k f_i(u)u_i.\]
It is clear that $P$ is a linear projection. Let us estimate its
norm. Let $f\in\lip_0(M)$ be such that $||f||=1$ and
$f(P(u))=||P(u)||$. Then
\[||P(u)||=\sum_{i=1}^k f_i(u)f(u_i)\le \left\|\sum_{i=1}^k
f(u_i)f_i\right\|\cdot||u||\le 2\max_i|f(u_i)|\cdot||u||\le
2||u||.\]

It remains to recall that the construction is such that
$k\ge\card(M)/2$.
\end{proof}

\begin{problem} Whether the constant $2$ in the statement ``$2$-complemented $2$-iso\-mor\-phic'' of Theorem \ref{T:HalfDimL1} is sharp?
\end{problem}

It is not surprising that Theorem \ref{T:HalfDimL1} can be
sharpened for some classes of graphs. In Theorem \ref{T:PropDiam}
we sharpen it for the diamond graphs.

It is natural to ask: How and when can we go beyond
half-dimensional subspace? It is easy to see that the following
result  can be proved on the same lines as Theorem
\ref{T:HalfDimL1}.

\begin{theorem}\label{T:LargeL1}
Let $M$ be a finite metric space and $\{y_i\}_{i=1}^k$ be a
sequence in it such that $M\backslash\{y_i\}_{i=1}^k$ is nonempty.
Let $d_i=d(y_i,(M\backslash\{y_i\}_{i=1}^k))$ and
\begin{equation}\label{E:IPConst} C=\max\left\{\max_{i\ne j}\frac{d_i+d_j}{d(y_i,y_j)},1\right\}.\end{equation} Then $\lf(M)$
contains as $C$-complemented subspace which is $C$-isomorphic to
$\ell_1^k$ and $\lip_0(M)$ contains a $C$-complemented subspace
which is $C$-isomorphic to $\ell_\infty^k$.
\end{theorem}

\begin{corollary}\label{C:LargeSub} If $M$ is a connected unweighted graph with $n$ vertices,
then for every $p\in\mathbb{N}$ with $p\le\diam(M)+1$ the space
$\lf(M)$ contains a subspace of dimension $d\ge
n\left(\frac{p-1}p\right)$ which is $4p$-complemented and is
$4p$-isomorphic to $\ell_1^d$, and $\lip_0(M)$ contains a
$4p$-complemented subspace which is $4p$-isomorphic to
$\ell_\infty^d$. If $p>\diam(M)$, we have the inequality
$d_{BM}(\lf(M),\ell_1^{n-1})\le 2p$ for the Banach-Mazur distance.
\end{corollary}

\begin{proof} Let $O$ be one of the vertices of $M$ for which $\max_{v\in M}d_M(O,v)=\diam(M)$. Assume that $p\le\diam(M)+1$. Consider partition
$M=\cup_{i=0}^{p-1}M_i$, where $M_i$ is the set of vertices in $M$
whose distance to $O$ is $i\pmod p$. The assumption
$p\le\diam(M)+1$ implies that all sets $M_i$ are nonempty. One of
the sets $\{M_i\}_{i=0}^{p-1}$ has cardinality $\le\frac{n}p$. Let
$\{y_i\}_{i=1}^k$ be the complement of this set. Its cardinality,
which we denote by $d$, is at least $n\left(\frac{p-1}p\right)$.
On the other hand, it is clear that $d_i\le 2p$ ($d_i$ is defined
in Theorem \ref{T:LargeL1}). Thus the constant $C$ defined in
\eqref{E:IPConst} is $\le 4p$. The conclusion follows.

The last statement is true because $p\ge \diam(M)$ implies that
the space $M$ is $2p$-bilipschitz equivalent to the graph
$K_{1,n-1}$ with its graph distance, and $\lf(K_{1,n-1})$ is
isometric to $\ell_1^{n-1}$ by Proposition \ref{P:Tree}.
\end{proof}

For some graphs the estimates of Theorems \ref{T:HalfDimL1},
\ref{T:LargeL1} and Corollary \ref{C:LargeSub} can be improved
significantly. It is interesting that this can be done even in the
case of diamond graphs $\{D_n\}$, while $\lf(D_n)$ are far from
$\ell_1^{d(n)}$ of the corresponding dimension, see Corollary
\ref{C:DiamLargeL1} below and Theorem \ref{C:LFDnNotL1}.
\medskip

\begin{theorem}\label{T:PropDiam} $\lf(D_n)$ contains a
1-complemented isometric copy of  $\ell_1^k$ with $k=2\cdot
4^{n-1}$.
\end{theorem}

Note that for large $n$ the number $2\cdot 4^{n-1}$ is very close
to $\frac34\,|V(D_n)|$, see page \pageref{P:(b)}.

\begin{proof} We use an argument similar to the argument of Theorem
\ref{T:HalfDimL1} with the following choice of $\{y_i\}_{i=1}^k$:
the vertices $\{y_i\}_{i=1}^k$ are the vertices added to the graph
in the last step. Formula {\bf (b)} on page~\pageref{P:(b)}
implies that $k=2\cdot 4^{n-1}$. The vertex $x_i$ is chosen to be
one of the (two) closest to $y_i$ vertices in $D_n$. In this case
$d(x_i,y_i)=1$ and $d(y_i,y_j)\ge2$ for $i\ne j$. Hence the same
argument as in Theorem \ref{T:HalfDimL1} leads to a subspace
isometric to $\ell_1^k$ and $1$-complemented.
\end{proof}

Corollary of Theorem \ref{T:LargeL1} for diamonds:

\begin{corollary}\label{C:DiamLargeL1} For each $m<n$ the space $\lf(D_n)$ contains a $C$-complemented
$C$-isomorphic to $\ell_1^k$ subspace with $C=2^{n-m}$ and
$k=2(1+\sum_{i=0}^{n-1}4^i)-2\cdot 4^{m-1}$.
\end{corollary}

Note that the codimension of the subspace does not exceed
$\frac3{4^{n-m+1}}|V(D_n)|$.

\begin{proof} Consider in $D_n$ the
subset $A_{n,m}$ of vertices which were added when $D_m$ was
created. The equality {\bf (b)} on page~\pageref{P:(b)} implies
that the cardinality of $A_{n,m}$ is $2\cdot 4^{m-1}$. It is also
easy to see that the distance from any other vertex to this set
does not exceed $2^{n-m-1}$. Define $\{y_i\}_{i=1}^k$ as
$V(D_n)\backslash A_{n,m}$.

Then the constant $C$ defined in \eqref{E:IPConst} does not exceed
$2^{n-m}$ and $k=2(1+\sum_{i=0}^{n-1}4^i)-2\cdot
4^{m-1}$.\end{proof}

Results of this section lead us to a suspicion that Lipschitz free
space of dimension $n$ cannot be ``too far'' from $\ell_1^n$ in
the Banach-Mazur distance. In this connection we ask

\begin{problem}\label{P:FarFromell1} Estimate the maximal possible Banach-Mazur distance between $\ell_1^n$ and
a Lipschitz free space of dimension $n$.
\end{problem}

So far all known estimates for the Banach-Mazur distance
$d_{BM}(\lf(M),\ell_1^n)$ (where $n=|M|-1$) from below are at most
logarithmic in $n$. We know two cases in which logarithmic
estimates from below hold. One case is the case of diamond graphs
(if we use estimates based on the theory of Haar functions), see
Theorems \ref{C:LFDnNotL1} and \ref{T:LFmultL1}.

The second case is the case where $M$ itself has large
$\ell_1$-distortion. It is well-known that the $\ell_1$-distortion
of $n$-vertex expanders is of order $\log n$, see \cite{LLR95}.
Another example with $\log n$-distortion was given in
\cite[Corollary 1]{KN06} (see also \cite[Section 4.2]{Ost13}).
Bourgain \cite{Bou85} proved that the $\ell_1$-distortion of an
$n$-element metric space can be estimated from above by $C\log n$.
Therefore on these lines we cannot get lower estimates for
$d_{BM}(\lf(M),\ell_1^n)$ of higher than logarithmic order.

Observe that if $M$ is an expander, then
$d_{BM}(\lf(M),\ell_1^n)\le C\log n$ because expanders have
diameter of order $\log n$ and thus are $C\log n$-bilipschitz
equivalent to the tree $K_{1,n}$.
\medskip

Corollary \ref{C:LargeSub} allows us to get an estimate
(Proposition \ref{P:ell1}) for $d_{BM}(\lf(M),\ell_1^n)$ from
above in the case where $M$ is an unweighted finite graph, which
is slightly better than the estimate $d_{BM}(X_n,\ell_1^n)$ for a
general $n$-dimensional Banach space $X_n$.

Let us recall known estimates for
$\mathcal{D}_n:=\max\{d_{BM}(X_n,\ell_1^n):~\dim X_n=n\}$:

\[n^{\frac59}\log^{-C} n\le \mathcal{D}_n\le (2n)^{\frac56}\]
for some absolute constant $0<C<\infty$. The lower estimate is due
to Tikhomirov \cite{Tik18}; it is an improvement of the previous
estimate of \cite{Sza90}. The upper estimate in this form is due
to Youssef \cite{You14}; it is an improvement of previous
estimates of \cite{BS88}, \cite{ST89}, and \cite{Gia95}.

\begin{proposition}\label{P:ell1} If $M$ be an unweighted connected graph with $n+1$ vertices (endowed with its graph distance),
then $d_{BM}(\lf(M),\ell_1^n)\le
Cn^{\frac8{11}}$.
\end{proposition}

\begin{proof} We shall work with the dual space, that is, we will show
that \[d_{BM}(\lip_0(M),\ell_\infty^n)\le Cn^{\frac8{11}}.\]

By Corollary \ref{C:LargeSub} we can find elements
$f_i\in\lip_0(M)$ such that
\[\max_i|\alpha_i|\le\|\sum\alpha_if_i\|\le 4p \max|\alpha_i|\]
where the codimension of the subspace $F$ spanned by $\{f_i\}$ is
$k\le \frac{n}p$ provided $p\le\diam(M)+1$.  By an easy corollary
of the Kadets-Snobar \cite{KS71} theorem every subspace of
codimension $m$ of a finite-dimensional normed space is the range
of a projection of norm at most $\sqrt{m} + 1$. Hence we can  find
a projection $P$ of norm at most $2\sqrt{\frac np}$ onto $F$. By
the result of \cite{You14}, \[d_{BM}(\ker
P,\ell_\infty^k)\le\left(2\frac np\right)^{\frac 56}\] Therefore
we can find a sequence $\{g_i\}$ in $\ker P$ such that
\[\max_i|\beta_i|\le\left\|\sum \beta_ig_i\right\|\le\left(2\frac np\right)^{\frac
56}\max_i|\beta_i|.\]

We have
\[\frac1{4\sqrt{n/p}}\max_{i,j}(|\alpha_i|,|\beta_j|)\le\left\|\sum\alpha_if_i+\sum\beta_jg_j\right\|\le
(4p+(2n/p)^{\frac56})\max_{i,j}(|\alpha_i|,|\beta_j|).\]

So the Banach-Mazur distance $d_{BM}(\lip_0(M),\ell_\infty^n)$ can
be estimated from above by
\[4\sqrt{\frac np}\cdot
\left(4p+\left(2\,\frac{n}{p}\right)^{\frac56}\right)\]

Pick $p=n^{\frac 5{11}}$. We get
$d_{BM}(\lip_0(M),\ell_\infty^n)\le cn^{\frac8{11}}$ either by
applying the argument above if $n^{\frac 5{11}}\le\diam(M)+1$, or
by using the final statement of Corollary \ref{C:LargeSub}
otherwise.
\end{proof}

\section{Proof for general recursive
families}\label{S:Recur}

The goal of this section is to show that if the graph $B$
satisfies the conditions listed in Section \ref{S:CondB}, then the
Banach-Mazur distances between the Lipschitz free spaces on $B_n$
(see Definition \ref{D:B_n}) and the spaces $\ell_1^{d(n)}$ of the
corresponding dimensions tend to $\infty$. See Theorem \ref{T:B_n}
for the statement of the result.

\begin{note}\label{N:BijVsIso} It is clear that each bijection $g$
on the edge set of a graph $G$ induces an isometry on the space
$\ell_1(E(G))$ given by
\begin{equation}\label{E:BijVsIso} f\mapsto h \Leftrightarrow
h(e)=f(g^{-1}e)\quad f,h\in \ell_1(E(G)), e\in E(G).
\end{equation}
With some abuse of notation we shall keep the notation $g$ for
this isometry.
\end{note}

\subsection{Conditions on $B$}\label{S:CondB}

The conditions below are not independent. Our goal is to list all
conditions which we use.

\begin{enumerate}

\item\label{I:1} Each edge is contained in a geodesic (a shortest
path) of even length joining the bottom and the top. Each path
joining the top and the bottom is geodesic.

\item\label{I:Dir} Each edge is directed to the vertex with the
smaller distance to the top. The cycle space $Z(B)$ is constructed
using this orientation of $B$. Each directed cycle in $B$ is a
union of two paths which are pieces of geodesics joining the top
and the bottom. On one of these paths the direction on the cycle
coincides with the direction in $B$, on the other it is opposite.

\item\label{I:Vert} The (underlying) graph $B$ has an automorphism
$v$ which interchanges top and bottom vertices. We say that $v$ is
a {\it vertical automorphism} of $B$. ({\it Underlying} here means
that the automorphism does not respect directions of edges.)

\item\label{I:BasInvv} The automorphism $v$ can be chosen in such
a way that each element of $Z(B)$ is a fixed point of $v$.

\item\label{I:Delta} Let $D$ be the distance between the bottom
and the top in $B$. Consider the vector
\begin{equation}\label{E:Delta}\Delta=\frac1{DK}\sum_p\1_p\in\ell_1(E(B)),\end{equation}
where $K$ is the number of distinct geodesics joining the bottom
and the top in $B$, and $\1_p$ is the indicator function of a
bottom-top geodesic, and the sum is over all distinct bottom-top
geodesics.

It is easy to see that the map $E_n$ $(n=0,1,2,\dots)$ which maps
the indicator function $\1_e$ of an edge $e$ onto $\Delta$ in the
copy of $B$ which replaces $e$ extends to an isometric embedding
of $\ell_1(E(B_n))$ into $\ell_1(E(B_{n+1}))$, and that $E_n$ maps
$Z(B_n)$ into $Z(B_{n+1})$.  We introduce the function $c(B)$ in
$\ell_1(E(B))=\ell_1(E(B_1))$ as the function whose absolute value
is $E_0(\1_e)$ (where $e$ is the only edge of $B_0$), and the
signs are positive for edges which are closer to the top and
negative for edges which are closer to the bottom (recall that
each edge belongs to a geodesic of even length joining the top and
the bottom). One of the {\bf conditions} on these maps is:
$v(c(B))=-c(B)$ (see Note \ref{N:BijVsIso}), this condition
actually follows from other conditions. Another condition is in
item \ref{I:NoFixCyclesButc(B)}.

\item\label{I:NoFixCyclesButc(B)} The collection $\mathcal{H}$ of
all automorphisms of $B$ for which the top and the bottom are
fixed points satisfies two conditions. First is that the
corresponding subgroup of isometries of $\ell_1(E(B))$ has no
fixed points in the cycle space $Z(B)$ except $0$. Second is that
the function $c(B)$ is a common fixed point of all elements of
$\mathcal{H}$. We call automorphisms of $\mathcal{H}$ {\it
horizontal}.

\item\label{I:z} The cycle space of $B$ is nontrivial. This is
equivalent to the existence of two distinct bottom-top geodesics,
and this is, in turn, equivalent to the fact that
$\frac1D\1_p-\Delta\ne 0$ for any bottom-top geodesic $p$. We pick
a bottom-top geodesic $p$ for which the $\ell_1$-norm
$\frac1D\1_p-\Delta\ne 0$ is maximized, denote this difference by
$d(B)$ and its norm in $\ell_1(E(B))$ by $\alpha$. Observe that
$d(B)\in Z(B)$.

\end{enumerate}

It is worth mentioning that the graphs $\{B_n\}$, $n\ge 1$,
inherit some properties of the graph $B=B_1$.
\medskip

\noindent{\bf (A)} Graphs $B_n$ have properties of items \ref{I:1}
and \ref{I:Dir}.
\medskip

Only the last condition in item \ref{I:Dir} requires verification.
This can be done using induction. We have assumed this condition
for $B_1$. Suppose that holds for $B_{n-1}$. Consider a directed
cycle in $B_n$. By \eqref{E:DiffProd} we have $B_n=B\oslash
B_{n-1}$. If the cycle is contained in one of the copies of
$B_{n-1}$, we are done by the induction hypothesis. If the cycle
is not contained in any of $B_{n-1}$, then it can be obtained
replacing each edge in the corresponding cycle in $B_1$ by a
top-bottom path in the corresponding copy of $B_{n-1}$ (see item
\ref{I:1}). The conclusion follows if we recall how edges of $B_n$
are oriented, see Definition \ref{D:Comp}.\medskip

\noindent{\bf (B)} The underlying graph of $B_n$ has an
automorphism $v_n$ which interchanges top and bottom vertices.
\medskip

This can be proved by induction:

\begin{itemize}

\item For $B_1=B$ this is an assumption of item \ref{I:Vert}.

\item Suppose that this is true for $B_{n-1}$ and $v_{n-1}$ is the
corresponding automorphism. By \eqref{E:DiffProd} we have
$B_n=B\oslash B_{n-1}$. We consider the bijection of the edge set
of $B_n$ designed in the following way:

\item If $v$ maps an edge $uw$ to an edge $\hat u\hat w$, with $u$
and $\hat u$ being closer to the bottom of $B$, we map $B_{n-1}$
corresponding to the edge $uw$ onto $B_{n-1}$ corresponding to
$\hat u\hat w$ ``upside down'', that is, using $v_{n-1}$.

\item It is easy to see that we get an automorphism of $B_n$,
which interchanges the top and the bottom. We denote this
automorphism by $v_n$.

\end{itemize}

\subsection{The main result}

Our main result on families $\{B_n\}$:

\begin{theorem}\label{T:B_n} If the directed graph $B$ satisfies the conditions in items \ref{I:1}--\ref{I:z} listed above, and $\{B_n\}_{n=0}^\infty$ are constructed
according to Definition \ref{D:B_n}, then
\[d_{BM}(\lf(B_n), \ell_1^{d(n)})\ge\frac{cn}{\ln n}\] for $n\ge 2$ and some absolute constant $c>0$, where $d(n)$ is the dimension of
$\lf(B_n)$.
\end{theorem}

To prove Theorem \ref{T:B_n} we need several lemmas. The final
step in the proof is presented on page \pageref{Page:EndPrf}.

\begin{lemma}\label{L:RedProj} To prove Theorem \ref{T:B_n} it
suffices to show that the relative projection constants of
$Z(B_n)$ in $\ell_1(E(B_n))$ satisfy
\[\lambda(Z(B_n), \ell_1(E(B_n)))\ge\frac{cn}{\ln n}\]
for $n\ge 2$ and some absolute constant $c>0$.
\end{lemma}

\begin{proof} This lemma is a consequence of the following
well-known fact.\end{proof}

\begin{fact}\label{F:Lift} If a quotient $X/Y$ is such that the
Banach-Mazur distance satisfies \[d_{BM}(X/Y,\ell_1(\Gamma))\le
C,\] then $\lambda(Y,X)\le (1+C)$.
\end{fact}

\begin{proof}[Proof of Fact \ref{F:Lift}] Denote by $Q: X\to X/Y$ the quotient map. Let $T:\ell_1(\Gamma)\to X/Y$ be such that $||T||< C+\ep$,
$||T^{-1}||\le 1$. By the {\it lifting property} of
$\ell_1(\Gamma)$ (see \cite[pp.~107--108]{LT77}), there is a
linear operator $\widehat T: \ell_1(\Gamma)\to X$ such that
$||\widehat T||<C+\ep$ and $Q\widehat T=T$. Then the operator $(I
- \widehat T T^{-1} Q)$ is a projection of $X$ onto $Y$, and its
norm is $<(1+C+\ep)$, the conclusion follows.
\end{proof}

\subsection{Cycle-preserving bijections of $B_n$}\label{S:CycBij}

For each $n\in\mathbb{N}$ we introduce $G_n$ as the group of all
{\it cycle-preserving bijections} of $E(B_n)$ (we consider
undirected edges) satisfying the additional condition: the edge
set of any path joining the top and the bottom of $B_n$ is mapped
onto the edge set of a path joining the top and the bottom of
$B_n$. By a {\it cycle-preserving bijection} we mean a bijection
which maps the edge-set of any cycle to an edge-set of a cycle (we
do not pay attention to directions of edges). It is clear that
$G_n$ is a finite group.
\medskip

The representation \eqref{E:DiffProd} shows that for each $1\le
k\le n-1$ the graph $B_n$ is a union of edge-disjoint copies of
$B_k$. It is clear that bijections of $E(B_n)$ which leave all
these copies of $E(B_k)$ invariant, and whose restrictions to them
are contained in $G_k$, belong to $G_n$.\medskip

The groups $G_n$ lead in a natural way (see Note \ref{N:BijVsIso})
to subgroups of the group of isometries of $\ell_1(E(B_n))$. An
important observation is that the subgroup corresponding to $G_n$
leaves the cycle space $Z(B_n)$ invariant.
\medskip

This observation can be shown as follows. By statement {\bf (A)}
each directed cycle in $B_n$ is a union of two pieces, $C_1$ and
$C_2$, of geodesics (going up and going down). Thus there are
well-defined notions of the {\it top} (and {\it bottom}) {\it of
the cycle} - the vertex of the cycle nearest to the top (bottom)
of $B_n$. We join them to the top and bottom of $B_n$,
respectively, using pieces of geodesics $P_b$ and $P_t$. Then both
the concatenation $P_bC_1P_t$ and $P_bC_2P_t$ are paths joining
the bottom and the top of $B_n$. Therefore the additional
condition on cycle-preserving bijections implies that the edge
sets of $P_bC_1P_t$ and $P_bC_2P_t$ are edge-sets of bottom-top
paths in $B_n$. Also the image of the edge set of the cyclic
concatenation of $C_1C_2$ is an edge set of a cycle. It is easy to
see that these conditions together imply that the images of $C_1$
and $C_2$ are parts of bottom-top geodesics. Hence the image of
$C$ is in the cycle space.

Observe that $G_1$ contains both $\mathcal{H}$ and the vertical
automorphism $v$, and thus the group generated by
$\mathcal{H}\cup\{v\}$.

\subsection{Gr\"unbaum-Rudin-Andrew-type
averaging}\label{S:GRAAppr}

Usage of the averages of the following type for estimates of
projection constants goes back at least to Gr\"unbaum \cite{Gru60}
and Rudin \cite{Rud62}. It was used in a similar to the present
context by Andrew \cite{And78}.

Let $P$ be any linear projection of $\ell_1(E(B_n))$ onto
$Z(B_n)$. Since $G_n$ is a finite group, which can be regarded as
a group of isometries of $\ell_1(E(B_n))$, the following operator
is well-defined
\begin{equation}\label{E:DefPGn}
P_{G_n}:=\frac1{|G_n|}\sum_{g\in G_n}g^{-1}Pg,\end{equation} is
also projection on $Z(B_n)$, and $||P_{G_n}||\le ||P||$. It is
easy to check that $P_{G_n}$ has the following important property:
\begin{equation}\label{E:CommWithGn}
\forall g\in G_n\quad P_{G_n}g=gP_{G_n}
\end{equation}

We call a projection satisfying \eqref{E:CommWithGn} {\it
invariant} with respect to $G_n$.
\medskip

The new twist in the usage of the method in our paper (see
Sections \ref{S:Annihil} and \ref{S:Combin}) is that we use it in
situations where the invariant projection is not unique. Namely we
observe that although in some situations which we consider
$P_{G_n}$ obtained by formula \eqref{E:DefPGn} is not unique (see
Section \ref{S:NonUnLaakso}), it is possible to show, see Lemma
\ref{L:Annihil}, that there is a collection of vectors in
$\ell_1(E(B_n))$ which are mapped to $0$ by any $P_{G_n}$
satisfying \eqref{E:CommWithGn}. This allows us to show that in
the considered cases $||P_{G_n}||$ grows indefinitely as
$n\to\infty$, see Section \ref{S:Combin}; and to get the estimate
stated in Lemma \ref{L:RedProj}.

\subsection{Bases in the spaces $Z(B_n)$}

We need to find a basis $S_n$ in the cycle space $Z(B_n)$, $n\ge
1$. Each of the bases which we pick will satisfy the following
conditions.

\begin{itemize}

\item[{\bf (i)}] Each element is either a fixed point of $v_n$, or
is supported on a copy of some $B_k$, $1\le k\le n-1$, and is an
element of the corresponding $S_k$.

\item[{\bf (ii)}] If an element is a fixed point of $v_n$, then
its restriction to any $B_k$, $1\le k\le n-1$, is a multiple of
$\Delta_k$, and thus is a fixed point of the corresponding $G_k$
(see the discussion next to \eqref{E:Delta_m} below). This
condition is void if $n=1$.

\end{itemize}

\medskip

Since $B_1=B$, we let $S_1$ be any basis in $Z(B)$. The conditions
{\bf (i)} and {\bf (ii)} are trivially satisfied, see item
\ref{I:BasInvv} in Section \ref{S:CondB}.
\medskip

Let $e\in E(B_k)$. It is easy to verify that the function
$E_{m+k-1}E_{m+k-2}\dots E_k \1_e\in \ell_1(E(B_{m+k}))$, which is
supported on a copy of $B_m$ which evolved from $e$, can be
written (similarly to \eqref{E:Delta}) as
\begin{equation}\label{E:Delta_m}\Delta_m:=\frac1{D_{m}K_m}\sum_p\1_p,\end{equation} where $K_m$ is the
number of distinct geodesics joining the bottom and the top of the
copy of $B_m$ mentioned above, $\1_p$ is the indicator function of
a bottom-top geodesic in $B_m$, and the sum is over all distinct
bottom-top geodesics. It is easy to see that $\Delta_1=\Delta$.

\begin{observation}\label{O:FixedDelta_m} Since any element of $G_m$ maps bijectively bottom-top paths
in $B_m$, we see that the function $\Delta_m$ is the fixed point
of any element of $G_m$ considered as acting on the considered
copy of $B_m$.
\end{observation}

Now we pick basis in $Z(B_n)$ assuming that we already picked a
basis $S_{n-1}$ in $Z(B_{n-1})$. The basis consists of two types
of vectors:

\begin{itemize}

\item[{\bf (I)}] Vectors which were already picked for $S_{n-1}$
in one of the copies on $B_{n-1}$ in $B_n$. Recall that
$B_n=B_1\oslash B_{n-1}$, see \eqref{E:DiffProd}.

\item[{\bf (II)}]  For each $f\in S_1$ we consider the following
function on $B_n=B_1\oslash B_{n-1}$: its restriction to each of
the copies of $B_{n-1}$ is a product of the corresponding
$\Delta_{n-1}$ and the value of $f$ on the edge from which the
considered copy of $B_{n-1}$ has evolved.

\end{itemize}

\begin{observation}\label{O:TypeII} Any vector of type {\bf (II)}
is a fixed point of any $G_{n-1}$. The same holds for any $G_k$,
$1\le k\le n-1$ corresponding to $B_n=B_{n-k}\oslash B_k$ and
acting on one of the copies of $B_k$. For the second statement we
need to observe that the restriction of $\Delta_{n-1}$ to $B_k$ is
a multiple of the corresponding $\Delta_k$.
\end{observation}

First we need to show that conditions {\bf (i)} and {\bf (ii)} are
satisfied. It is easy to see that the only statement requiring a
proof is the fact that the function constructed in the previous
paragraph is a fixed point for $v_n$.\medskip

To see this we observe that the values of $f$ corresponding to
copies of $B_{n-1}$ which are mapped one onto another by $v_n$ are
equal because $f$ is a fixed point of $v$ and by construction of
$v_n$. Thus we get the desired conclusion.

\begin{lemma}\label{L:Basis} The set $S_n$ is a basis of the
linear space $Z(B_n)$.
\end{lemma}

\begin{proof} We use induction. For $n=1$ this is true according
to our choice. Suppose that the statement holds for $n-1$, and
show that this implies it for $n$. We need to show two things:
completeness and linear independence.\medskip

\noindent{\bf Completeness:} (1) If a cycle is contained in one of
the $B_{n-1}$, then it is contained in the linear span of the
corresponding $S_{n-1}$ by the induction hypothesis, and we are
done because $S_{n}$ contains that $S_{n-1}$.

(2) If a cycle $C$ is not contained in any of $B_{n-1}$, then,
after collapsing each of $B_{n-1}$ to the edge of $B_1$ from which
it evolved (according to $B_n=B_1\oslash B_{n-1}$), we get a
nontrivial cycle $\widehat C$ in $B_1$. This cycle is a linear
combination of cycles of $S_1$ (since $S_1$ is a basis in
$Z(B_1)$), so $\widehat C=\sum\gamma_is_i$ for some
$\gamma_i\in\mathbb{R}$ and $s_i\in S_1$. Denote the composition
$E_{n-1}\dots E_1$ by $\mathcal{E}_{n-1}$. We have

\[\mathcal{E}_{n-1}\widehat C=\sum\gamma_i\mathcal{E}_{n-1}s_i.\]

The description of the type {\bf (II)} vectors implies that
vectors $\mathcal{E}_{n-1}s_i$ are elements of $S_n$. Therefore it
remains to analyze the difference $C-\mathcal{E}_{n-1}\widehat C$.
\medskip

For each $B_{n-1}$ in $B_n$ (according to $B_n=B_1\oslash
B_{n-1}$) one of the following is true:

\begin{itemize}

\item There are no edges of $C$ and no edges of
$\mathcal{E}_{n-1}\widehat C$ in it.

\item There is a path $p$ from the bottom to top of $B_{n-1}$
which is contained in $C$, and the corresponding part of
$\mathcal{E}_{n-1}\widehat C$ is $\Delta_{n-1}$.

\end{itemize}

It remains to observe that that $\Delta_{n-1}-\1_p$ belongs to
$Z(B_{n-1})$ (follows from the formula for $\Delta_{n-1}$). Thus
the difference $C-\mathcal{E}_{n-1}\widehat C$ can be written as a
sum of elements of $S_{n-1}$ for those $B_{n-1}$ which contain
nontrivial sub-paths of $C$. As a conclusion we get that $C$ is in
the linear span of $S_n$.
\medskip

\noindent{\bf Linear Independence.} It is clear that a nontrivial
linear combination of vectors of type {\bf (I)} cannot be equal to
$0$, since $S_{n-1}$ are linearly independent and $B_{n-1}$ are
edge-disjoint.
\medskip

For this reason to prove linear independence it is enough to show
that a nontrivial linear combination containing vectors of type
{\bf (II)} cannot be $0$.
\medskip

We split a linear combination as $a+b$, where $a$ is a linear
combination of vectors of type {\bf (I)}, and $b$ is a linear
combination of vectors of type {\bf (II)}. Observe that $b$ can be
obtained in the following way: we consider a non-zero vector
vector in $Z(B_1)$, and then replace each $\1_e$ used in this
vector, by the corresponding $\Delta_{n-1}$. Because of this the
restriction of $b$ to at least one of $B_{n-1}$ does not belong to
$Z(B_{n-1})$. Hence $a+b$, restricted to that $B_{n-1}$ is
nonzero, and we are done.
\end{proof}

\subsection{Invariant projections annihilate functions $c(B_1)$
and their images under $E_k$}\label{S:Annihil}

\begin{lemma}\label{L:Annihil} The projection $P_{G_n}$ annihilates all of the functions
of the form $c(B_1)$ for some $B_1$ in $B_n$, and functions which
are obtained from $c(B_1)$ by repeated applications of $E_k$.
\end{lemma}

\begin{proof} Let $f$ be some function of the described form in
$\ell_1(E(B_n))$ and let $B_m$, $m\le n$, be a subgraph of $B_n$
supporting $f$. It is easy to see that the absolute value of $f$
is equal to the function $\Delta_m$ described in
\eqref{E:Delta_m}, and that $f$ is positive on edges which are
closer to the top of $B_m$ and negative on the edges which closer
to the bottom of $B_m$.\medskip

Suppose, contrary to the statement of the Lemma, that
$P_{G_n}f=q\ne 0$. Since $q\in Z(B_n)$, it is a linear combination
of vectors described in {\bf (I)--(II)}.\medskip

It is clear that one of the following is true:

\begin{itemize}

\item[$(\le)$] One of the vectors of the basis described in {\bf
(I)--(II)}, present in the linear combination representing $q$,
belongs to $S_k$ with $k\le m$.

\item[$(>)$] All vectors of the basis described in {\bf
(I)--(II)}, present in the linear combination representing $q$,
belong to $S_k$ with $k>m$.

\end{itemize}

We show, that in each of these cases we get a contradiction with
the invariance of $P_{G_n}$.
\medskip

\noindent{\bf Case $(\le)$.} Assume that $k$ is the smallest
integer with this property. Since it is the smallest integer, all
basis elements with nonzero coefficients belonging to $S_k$ are of
type {\bf (II)}. Therefore they correspond to certain elements of
$S_1$, and their linear combination $\mu$ (as it is present in the
representation of $q$) corresponds to nonzero element $\tau$ of
$Z(B_1)$. By Condition \ref{I:NoFixCyclesButc(B)} (on $B$) there
exists a horizontal automorphism $g$ of $B_1$, such that $g\tau\ne
\tau$. Let us consider an automorphism $\hat g$ of $S_k$ {\it
induced} by $g$ in the following way. The automorphism $g$ is a
bijection of $E(B_1)$. In $B_k=B_1\oslash B_{k-1}$ we consider the
corresponding bijections of subgraphs $B_{k-1}$ which evolved from
those edges. It is clear that $\hat g\in G_k$, and that $\hat
g\mu\ne\mu$.
\medskip

On the other hand, it is clear that $\hat g f=f$. In the case
where $k=m$ this follows from the fact that $c(B)$ is a fixed
point of all horizontal automorphisms (Condition
\ref{I:NoFixCyclesButc(B)}). In the case where $k<m$, this follows
from Observation \ref{O:FixedDelta_m}. We get a contradiction with
the fact that $P_{G_n}$ is an invariant projection (see
\eqref{E:CommWithGn}), because
\[\begin{split}P_{G_n}f&=P_{G_n}\hat gf=\hat gP_{G_n}f=\hat gq=\hat
g(\mu+(q-\mu))\\&=\hat g\mu+\hat g(q-\mu)=\hat g(\mu)+(q-\mu)\ne
\mu+q-\mu=q,\end{split}\] where we used the fact that elements of
the basis $S_n$ used in the decomposition of $q-\mu$ are either
edge-disjoint with the copy of $B_k$ on which $\mu$ is supported,
or are proportional to $\Delta_k$ on that $B_k$. In either case
$(q-\mu)$ is a fixed point of $\hat g$.
\medskip

\noindent{\bf Case $(>)$.} In this case, by Observation
\ref{O:FixedDelta_m}, any function used in the decomposition of
$q$ with respect to the basis $S_n$ is a fixed point of $v_m$,
which was defined in {\bf (B)}.\medskip

On the other hand, $v_mf=-f$, by the definitions of $v_m$ and $f$.
This contradicts the fact that $P_{G_n}$ is an invariant
projection (see \eqref{E:CommWithGn}), because we get
\[-P_{G_n}f=P_{G_n}v_mf=v_mP_{G_n}f=v_mq=q=P_{G_n}f.\qedhere
\]
\end{proof}

\subsection{Combining everything}\label{S:Combin}

\begin{proof}[Proof of Theorem \ref{T:B_n}]\label{Page:EndPrf} Let us show, using Lemma \ref{L:RedProj}, that in order to prove Theorem
\ref{T:B_n} it suffices to show that for each $r\in \mathbb{N}$
there exists $n=n(r)\in\mathbb{N}$,  $C_r\in Z(B_n)$, and a linear
combination $A_r$ of vectors of the forms $c(B_1)$ and their
images under $\{E_n\}$ such that
\begin{equation}\label{E:Norm1}
||C_r+A_r||=1
\end{equation}
and
\begin{equation}\label{E:P_Large}
||C_r||\ge 1+\frac{\alpha (r-1)}{2},
\end{equation}
where $\alpha>0$ is the number introduced in item \ref{I:z} of
Section \ref{S:CondB}, and to find a suitable estimate for the
corresponding $n(r)$ in terms of $r$.\medskip

In fact, for every projection $P:\ell_1(E(B_n))\to Z(B_n)$ we get

\[||P||\ge ||P_{G_n}||\ge ||P_{G_n}(C_r+A_r)||\stackrel{\small\hbox{(Lemma
\ref{L:Annihil})}}{=}||C_r||\ge 1+\frac{\alpha(r-1)}2,\] which
proves the desired inequality for the projection constant.
\medskip

Case $r=1$. We let $C_1$ be any $\ell_1$-normalized element of
$S_1$ (use non-triviality);  $A_1=0$. Everything is obvious.

Inductive step. Suppose that we have already constructed $C_r$ and
$A_r$ in some $B_{n(r)}$.
\medskip

We apply $E_{n(r)}$ to $C_r+A_r$. Observe that $E_{n(r)}$ maps the
cycle space into the cycle space, and preserves the desired form
of the function $A_r$. Observe that $C_r+A_r$, as an element of
$\ell_1(E(B_{n(r)}))$ is a linear combination of edges. Therefore
$E_{n(r)}(C_r+A_r)$ is of the form $\displaystyle{\sum_{e\in
E(B_{n(r)})}a_{e,1}E_{n(r)}\1_e}$, where $a_{e,1}$ are real
numbers. The functions $E_{n(r)}\1_e$ are of the form $\Delta$
(see \eqref{E:Delta}), supported on different copies of $B_1$,
recall that
\begin{equation}\label{E:Dec}
B_{n(r)+1}=B_{n(r)}\oslash B_1 \end{equation} (see
\eqref{E:DiffProd}). We let $C_r^1=E_{n(r)}C_r$ and let
\[A_r^1=E_{n(r)}A_r+\sum_{e\in
E(B_{n(r)})}a_{e,1}c(B_1),\] where $c(B_1)$ is taken on the
corresponding copy of $B_1$, according to \eqref{E:Dec}. It is
easy to see that $||C_r^1+A_r^1||=1$, and its support is exactly
half (in many respects) of the support of $E_{n(r)}(C_r+A_r)$.
\medskip

We repeat the procedure for $C_r^1$ and $A_r^1$ instead of $C_r$
and $A_r$. We do this $t$ times, and get the functions which we
denote $C_r^t$ and $A_r^t$.

Some observations:

\begin{itemize}

\item The function $C_r^t$ is an image of $C_r$ under the
composition $E_{n(r)+t-1}\dots E_{n(r)}$.

\item The function $A_r^t$ is a linear combination of
$E_{n(r)+t-1}\dots E_{n(r)}A_r$ and images of $c(B_1)$ under some
compositions of $E_k$.

\end{itemize}

After that we do a somewhat different procedure. Namely we write
$E_{n(r)+t}(C_r^t+A_r^t)$ in the form $\sum_{e\in E(B_{n(r)+t})}
a_{e,t+1}E_{n(r)+t}\1_e$, where $a_{e,t+1}$ are real numbers. The
functions $E_{n(r)+t}\1_e$ are multiples of $\Delta$, supported on
different copies of $B_1$, recall that
$B_{n(r)+t+1}=B_{n(r)+t}\oslash B_1$. Now we let
\[A_{r+1}=E_{n(r)+t}(A_r^t)\]
and
\[C_{r+1}=E_{n(r)+t}(C_r^t)+\sum_{e\in E(B_{n(r)+t})}
a_{e,t+1}d(B),\] where $d(B)$ is the function defined in item
\ref{I:z} of Section \ref{S:CondB} and supported on the
corresponding copy of $B_1$.
\medskip

It is clear from the definition of $d(B)$ that
$||C_{r+1}+A_{r+1}||=1$. It is also clear that $C_{r+1}\in
Z(B_{n(r)+t+1})$, and $A_{r+1}$ is of the desired form.
\medskip

Observe that since $||\sum a_{e,t+1}1_e||=1$, we have $||\sum
a_{e,t+1}d(B)||=\alpha$ (see item \ref{I:z} in Section
\ref{S:CondB}). Our construction is such that the norm of the part
of $E_{n(r)+t}\dots E_{n(r)}C_r$ supported in the support of
$\sum_{e\in E(B_{n(r)+t})} a_{e,t+1}d(B)$ is
$\displaystyle{\frac1{2^t}||C_r||}$. Therefore if we pick $t$ in
such a way that
\begin{equation}\label{E:Need}\frac1{2^t}||C_r||<\frac{\alpha}4,\end{equation} we get

\[\begin{split}||C_{r+1}||&=\left\|E_{n(r)+t}\dots E_{n(r)}C_r+\sum_{e\in E(B_{n(r)+t})}
a_{e,t+1}d(B)\right\|\\&\ge ||E_{n(r)+t}\dots E_{n(r)}C_r||+
\left\|\sum_{e\in E(B_{n(r)+t})}
a_{e,t+1}d(B)\right\|-2\frac{\alpha}4\\&\stackrel{\eqref{E:P_Large}}{\ge}
1+\frac{\alpha (r-1)}2+\alpha-\frac{\alpha}2=1+\frac{\alpha
r}2.\end{split}\]

It remains to find an estimate for $n$ in terms of $r$. To achieve
the condition \eqref{E:Need} for $r\ge 2$ we need to pick $t\ge
C\ln r$ for some $C>0$.

This leads to the estimate $\lambda(Z(B_n),\ell_1(E(B_n)))\ge ck$
if $n\ge Ck\ln k$, where $c>0$, $C<\infty$ (the constants in these
statements do not have to be the same) and $d(n)$ is the dimension
of $\lf(B_n)$.
\medskip

It is easy to see that this estimate implies
\[ \lambda(Z(B_n),\ell_1(E(B_n)))\ge \frac{cn}{\ln n}.\qedhere\]
\end{proof}

\section{Consequences for multibranching diamond graphs and Laakso graphs}\label{S:Diamond}

Our next goal is to show that diamond graphs and Laakso graphs
satisfy the conditions listed in Section \ref{S:CondB}.

\subsection{Multibranching diamond graphs}

Condition \ref{I:1} in the case where $B$ is $K_{2,n}$, $n\ge 2$,
with the top and the bottom being the vertices of the part
containing two vertices is obvious.\medskip

Condition \ref{I:Dir} is clear.\medskip

For Condition \ref{I:Vert} we choose the automorphism in such a
way that it maps each bottom-top path onto itself.\medskip

With this choice of $v$ the Condition \ref{I:BasInvv} is easy to
check.\medskip

Condition \ref{I:Delta} is clearly satisfied.\medskip

Condition \ref{I:NoFixCyclesButc(B)}: A nonzero element of
$Z(K_{2,n})$ cannot be a fixed point of $\mathcal{H}$ because
(according to the directions chosen on edges) each non-zero
element of $Z(K_{2,n})$ has bottom-top paths on which the value is
positive and bottom-top paths on which the value is
negative.\medskip

The second part of Condition \ref{I:NoFixCyclesButc(B)} holds
because any horizontal automorphism maps edges which are closer to
the top (bottom) to edges which are closer to the top (bottom).
\medskip

Finally, Condition \ref{I:z} is satisfied because we consider
$n\ge 2$ and $s_1$ (element of the basis listed above) is an
example of a nontrivial cycle in $Z(K_{2,n})$. The value of
$\alpha$ is $\frac{2(n-1)}n$.

\subsection{Laakso graphs}

Condition \ref{I:1} in the case where $B$ is $\mathcal{L}_1$ with
the natural choice of the top and the bottom is obvious.\medskip

Condition \ref{I:Dir} is clear.\medskip

For Condition \ref{I:Vert} we choose the automorphism $v$ which
maps each bottom-top path onto itself.\medskip

Condition \ref{I:BasInvv}: There is only one cycle in
$\mathcal{L}_1$, it is obviously the fixed point of the chosen
automorphism of $\mathcal{L}_1$.\medskip

Condition \ref{I:Delta} is clearly satisfied.\medskip

The first part of Condition \ref{I:NoFixCyclesButc(B)} holds
because, by the choice of the directions of edges any nonzero
element of $Z(\mathcal{L}_1)$ has positive value on one side and
negative value on the other side, and thus is mapped onto its
negative by a nontrivial element of $\mathcal{H}$.\medskip

The second part of Condition \ref{I:NoFixCyclesButc(B)} holds
because any horizontal automorphism maps edges which are closer to
the top (bottom) to edges which are closer to the top (bottom).
\medskip

Condition \ref{I:z} is clearly satisfied. The value of $\alpha$ is
$\frac12$.

\subsection{Non-uniqueness of invariant projections of $\ell_1(E(\mathcal{L}_2))$
onto $Z(\mathcal{L}_2)$}\label{S:NonUnLaakso}

Our main goal in this section is to show that for Laakso graphs
there is no uniqueness of invariant projections. It is clear that
one of the invariant projections is the orthogonal projection onto
$Z(\mathcal{L}_2)$ in $\ell_2(E(\mathcal{L}_2))$. So it is enough
to construct an invariant projection which is not orthogonal.

\begin{proposition}\label{P:CutNotTo0} There exists an invariant linear projection of $\ell_2(E(\mathcal{L}_2))$ onto  $Z(\mathcal{L}_2)$
which is different from the orthogonal projection.
\end{proposition}

\begin{proof} We consider the following projection:
It is like the orthogonal projection on the top and bottom
``tails'' of $\mathcal{L}_2$ and is different only in the central
part. In the central part there are edges which belong to the
16-cycle only and edges which belong also to $4$-cycles.
\medskip

We introduce the following functions in $\ell_1(E(\mathcal{L}_2))$
supported on the central part of $\mathcal{L}_2$: \medskip

\noindent{\bf (1)} Indicator functions $\chi_C$ of cycles of
length $4$ (see \eqref{E:SignInd}) directed counterclockwise, so
they have values $1$ on the right sides and values $-1$ on the
left sides.
\medskip

\noindent{\bf (2)} The function $F=\frac{F_1+F_2}2$, where $F_1$
is the indicator function of the directed counterclockwise ``outer
cycle'' of length $16$ and $F_2$ is the indicator function of the
directed counterclockwise ``inner cycle'' of length $16$.
\medskip

We consider the projection which acts in the following way:
\smallskip

\noindent{\bf (a)} It maps each edge which is in the ``$16$-cycle
only''  to $\frac{\theta}8\,F$, where $\theta=1$ on the right half
and $\theta=-1$ on the left half.
\medskip

\noindent{\bf (b)} It maps each edge which is ``both in the
$16$-cycle and $4$-cycle'' onto the $\frac{\theta}4\chi_C$, where
$C$ is the corresponding $4$-cycle and $\theta=1$ on the right
side and $\theta=-1$ on the left side.
\medskip

It is straightforward to check that this projection is invariant
in the sense of \eqref{E:CommWithGn} and is different from the
orthogonal projection.
\end{proof}

\section{Lipschitz free spaces on diamond graphs -- more precise results using Haar
functions}\label{S:DiamHaar}

In this section we present an alternative self-contained proof of
our results for the binary diamond graphs $D_n$. This proof uses
the Haar system for $L_1[0,1]$ and makes an interesting connection
with some open problems concerning the even  levels of the Haar
system. At the end of this section we extend the proof to handle
the multi-branching diamond graphs as well.

We begin by reformulating  the definition of the binary diamond graphs in order to use the Haar system.  For $n \ge 2$, we shall consider $D_n$ as consisting of four copies of $D_{n-1}$, namely `top left', denoted $TL_n$, `bottom left', denoted $BL_n$, `bottom right', denoted $BR_n$,
and `top right', denoted $TR_n$. In this identification the bottom vertex of $TL_n$ coincides with the top vertex of $BL_n$, etc.

We shall identify the edge space of $D_n$, denoted $\ell_1(D_n)$,
with a certain subspace of $L_1[0,1]$. This identification is
recursive. We identify the edge vectors of $\ell_1(D_1)$ with the
functions $4\cdot\1_{((i-1)/4, i/4]}$ for $1 \le i \le 4$,
which are disjointly
supported unit vectors in $L_1[0,1]$, ordering the edges
$i=1,\dots,4$ counterclockwise from the top vertex. Now suppose
that $n \ge 2$ and that $\ell_1(D_{n-1})$ has been identified with
a subspace of $L_1[0,1]$. For a function $f\in L_1[0,1]$ we denote
by $Qf$ the function which is $0$ in $\left(\frac14,1\right]$ and
is given by $(Qf)(t)=4f(4t)$ for $t\in \left[0,\frac14\right]$. It
is clear that $Q$ is an isometric embedding of $L_1[0,1]$ into
itself. Then we identify $\ell_1(TL_n)$ with $Q(\ell_1(D_{n-1}))$,
and identify $\ell_1(BL_n)$, $\ell_1(BR_n)$, and $\ell_1(TR_n)$,
with
copies of $\ell_1(TL_n)$  translated by $\frac14,\frac12$, and $\frac34$ to the right, respectively. It follows that the edge
vectors of $\ell_1(D_n)$ are the functions $4^n
\cdot\1_{((i-1)/4^n, i/4^n]}$ for $1 \le i \le 4^n$, which are
disjointly supported unit vectors in $L_1[0,1]$.

Let us now determine the subspace of $L_1[0,1]$ which corresponds
under this identification to the cycle space of $D_n$, denoted
$Z(D_n)$. First, let us recall the definition of the Haar system
$(h_i)_{i\ge0}$. We define $h_0 := \1_{(0,1]}$, and for $n \ge 0$
and $0 \le i \le 2^n-1$,
$$h_{2^n + i} := \1_{(i/2^n, (2i+1)/2^{n+1}]} -
\1_{((2i+1)/2^{n+1}, (i+1)/2^n]}.$$ Let $H_n := \{ h_i \colon 2^n
\le i \le 2^{n+1}-1\}$ be the collection of all $2^n$ Haar
functions on the same level with support of length $2^{-n}$. Let
$e_n$ be the cycle vector corresponding to the `large outer cycle'
of $D_n$. To understand the pattern for $e_n$, first we calculate
$e_1$, $e_2$ and $e_3$. Clearly, $$e_1 = 4(\1_{[0,1/2]} -
\1_{[1/2,1]}) = 4h_1,$$
and \begin{equation}\label{eq: e_2} \begin{split} e_2 &= 16(\1_{[0,1/8]} + \1_{[2/8,3/8]} - \1_{[5/8,6/8]}- \1_{[7/8,1]})\\
& = 8(h_1 + h_4 + h_5 + h_6 + h_7)\\
&= 2e_1 + 8(\sum_{h \in A_2} h), \end{split}\end{equation} where $A_2 = \{ h \in H_2 \colon \operatorname{supp} h \subseteq \operatorname{supp} e_1\}$.
Note that
\begin{equation} \label{eq: e_3}
\begin{split} e_3 &= 64[ (\1_{[0,1/32]} + \1_{[2/32,3/32]}+ \1_{[8/32,9/32]} + \1_{[10/32,11/32]})\\ &\quad- (\1_{[21/32,22/32]} + \1_{[23/32,24/32]}+ \1_{[29/32,30/32]} + \1_{[31/32,1]})]\\
&= 16[(h_1 + h_4 + h_5 + h_6 + h_7)\\&\quad +2(h_{16} + h_{17} + h_{20} + h_{21}+h_{26} + h_{27} + h_{30} + h_{31})]\\
&= 2e_2 + 32(\sum_{h \in A_3} h),
\end{split}\end{equation} where $A_3 = \{ h \in H_4 \colon \operatorname{supp} h \subseteq \operatorname{supp} e_2\}$.
The passage from $e_{n-1}$ to $e_n$ in the general case is
analogous to the passage from $e_2$ to $e_3$ above and is given by
a procedure which we now describe. let $I$ be a maximal dyadic
subinterval of $\operatorname{supp} e_{n-1}$. Let $I_1$, $I_2$,
$I_3$, and $I_4$ be the first, second, third, and fourth quarters
of $I$ ordered from left to right. To get $e_n$ from $e_{n-1}$, if
$I$ is contained in the support of the \textit{positive} part of
$e_{n-1}$  then  we replace $\1_I$ in the expression for $e_{n-1}$
by $\1_{I_1} + \1_{I_3}$, and if $I$ is contained in the support
of the \textit{negative} part of $e_{n-1}$  then we replace
$-\1_I$ in the expression for $e_{n-1}$  by $-(\1_{I_2} +
\1_{I_4})$. Expressing $e_n$ in terms of Haar functions, it
follows, by analogy with \eqref{eq: e_2} and \eqref{eq: e_3}
above, that
\begin{equation} \label{eq: e_n} e_n = 2e_{n-1} + 2^{2n-1} \sum_{h \in A_n} h, \end{equation}
where $A_n = \{h \in H_{2n-2} \colon \operatorname{supp} h \subseteq \operatorname{supp} e_{n-1}\}$. Iterating \eqref{eq: e_n} and
recalling that $e_1 = 4h_1$, we get
\begin{equation} \label{eq: e_nexpression}  e_n - 2^{n+1} h_1 \in \operatorname{span} (\cup_{k=1}^{n-1} H_{2k}). \end{equation}
\begin{lemma} \label{thm: Z(D_n)} For all $n \ge 1$, $$ \ell_1(D_n) = \operatorname{span}( \{h_0\} \cup (\cup_{k=0}^{2n-1} H_k))$$ and
\begin{equation} \label{eq: Z(D_n)} Z(D_n) = \operatorname{span} (\cup_{k=0}^{n-1} H_{2k}). \end{equation}

\end{lemma} \begin{proof} The description of $\ell_1(D_n)$ follows from the observation above that the edge vectors of $\ell_1(D_n)$ are the functions $4^n \1_{[(i-1)/4^n, i/4^n]}$ for $1 \le i \le 4^n$.

We prove \eqref{eq: Z(D_n)} by induction. Note that
$$Z(D_1) = \operatorname{span}(\{h_1\}) = \operatorname{span}(H_0),$$
which verifies the base case $n=1$. So suppose that $n \ge 2$ and that the result holds for $n-1$. Note that \begin{equation}
\label{eq: span1}
Z(D_n) = \operatorname{span}(Z(TL_n) \cup Z(BL_n) \cup Z(BR_n) \cup Z(TR_n) \cup \{e_n\}). \end{equation}
Recall that $TL_n$, $BL_n$, $BR_n$, and $TR_n$ are translated and dilated copies of $D_{n-1}$ on the intervals $[(i-1)/4]$ for $1 \le i \le 4$.
Hence  $Z(TL_n)$, $Z(BL_n)$, $Z(BR_n)$, and $Z(TR_n)$ are translated and  dilated copies of $Z(D_{n-1})$ on the intervals $[(i-1)/4]$ for $1 \le i \le 4$. Applying the inductive hypothesis  to $Z(D_{n-1})$ it follows that \begin{equation} \label{eq: span2}
\operatorname{span}(Z(TL_n) \cup Z(BL_n) \cup Z(BR_n) \cup Z(TR_n)) = \operatorname{span} (\cup_{k=1}^{n-1} H_{2k}). \end{equation}
Finally, from \eqref{eq: span1},   \eqref{eq: span2} and \eqref{eq: e_nexpression}, we get \begin{align*}
Z(D_n) &= \operatorname{span} (\{e_n\} \cup(\cup_{k=1}^{n-1} H_{2k}))\\ &= \operatorname{span} (\{h_1\} \cup(\cup_{k=1}^{n-1} H_{2k}))\\
&= \operatorname{span} (\cup_{k=0}^{n-1} H_{2k}).
\end{align*} \end{proof}
\begin{remark} Note that $Z(D_n)$ has dimension $\sum_{k=0}^{2n-2} 4^k = (4^{2n-1}-1)/3$. This can also be seen directly
without using Lemma~\ref{thm: Z(D_n)} since \eqref{eq: span1}
clearly implies that $\operatorname{dim} Z(D_n) = 4
\operatorname{dim} Z(D_{n-1}) + 1$.  Using this observation that
the spaces have the same dimension,  it suffices to show that
$Z(D_n) \subseteq \operatorname{span} (\cup_{k=0}^{n-1} H_{2k})$,
which follows from \eqref{eq: e_nexpression} and  \eqref{eq:
span2}. Thus the proof can be concluded slightly
differently.\end{remark} Our next goal is to prove that $Z(D_n)$
is not well-complemented in $\ell_1(D_n)$. This essentially
follows from a result of Andrew \cite{And78}. (Note that the idea
of using the average over the group of isometries to estimate
norms of projections goes back at least to Gr\"unbaum \cite{Gru60}
and Rudin \cite{Rud62}.) For completeness we present a slight
generalization of Andrew's elegant argument. Let
 $X_n = \operatorname{span} ( \{h_i \colon 0 \le i \le 2^{n+1}-1\}) = \operatorname{span}(\{h_0\}\cup (\cup_{k=0}^n H_k))$. Let $(\cdot,\cdot)$ denote the usual inner product in $L_2[0,1]$. Orthogonality will refer to this inner product.

 Suppose $i \ge 1$ and that $h_i \in H_k$. Define a linear isomorphism $g_i \colon X_n \rightarrow X_n$ by
\begin{equation*}  (g_if)(t) = \begin{cases}  f(t), &t \notin \operatorname{supp} h_i,\\
f(t+ 2^{-k-1}), &t \in h_i^{-1}(1),\\
f(t- 2^{-k-1}), &t \in h_i^{-1}(-1) \end{cases}
 \end{equation*} for all $f \in X_n$. Suppose now that $\| \cdot \|$ is any norm on $X_n$ with the property that each $g_i$ acts as  a linear isometry of $(X_n, \|\cdot\|)$. For our purposes,  $\|\cdot\|$ will be the usual norm of $L_1[0,1]$ or of $L_\infty[0,1]$.
Let $G$ be the group of isometries generated by $(g_i)_{i \ge 1}$. Note that $G$ is finite.

In the next proposition it is convenient to set $H_{-1} :=
\{h_0\}$.

\begin{lemma} \label{prop: Andrew} Let $A$ be any nonempty subset of $\{-1,0,1,\dots,n\}$ and let $P$ be any linear projection on $(X_n, \|\cdot \|)$
with range $Y := \operatorname{span} (\cup_{k \in A} H_k)$. Then $\|P\| \ge \|P_Y\|$, where $P_Y$ is the orthogonal projection onto $Y$.
\end{lemma}

\begin{proof} Let $$ Q = \frac{1}{ \operatorname{card}(G)}\sum_{g \in G} g^{-1}Pg.$$ Clearly $\|Q\| \le \|P\|$. Moreover $Q$ is a projection onto
$Y$ since $g(Y) = Y$ for all $g \in G$. It suffices to show that $Q = P_Y$. The proof of this makes use of the following observations:
\begin{enumerate} \item $gQ = Qg$ for all $g \in G$;
  \item $g_i h_i = -h_i$ for all $i \ge 1$;
  \item $(g_i f, h_i) = -(f, h_i)$ for all $f \in X_n$ and for all $i \ge 1$;
  \item If  $0 \le i < j$ and $\operatorname{supp} h_j \subset \operatorname{supp} h_i$  then $(g_j f, h_i) = (f, h_i)$ for all $f \in X_n$;
  \item If $i > j \ge0$ or if  $h_i$ and $h_j$ are disjointly supported then  $g_i h_j = h_j$.
  \end{enumerate}
Suppose that $h_j \notin Y$. We have to show that $Qh_j =0$. Since $Q$ is a projection onto $Y$, it suffices to show that if $h_i \in Y$ then
$(Qh_j, h_i)=0$. If $0 \le i < j$ and $\operatorname{supp} h_j \subset \operatorname{supp} h_i$ then
\begin{align*} (Qh_j, h_i) &= (g_jQh_j, h_i) \qquad\text{(by ($4$))}\\
&= (Qg_jh_j,h_i) \qquad\text{(by ($1$))}\\
&= -(Qh_j, h_i) \qquad \text{(by ($2$))}.\end{align*}
Hence $(Qh_j, h_i)=0$ in this case. Now suppose that $i > j\ge 0$ or that $h_i$ and $h_j$ are disjointly supported. Then
\begin{align*}(Qh_j, h_i) &= (Qg_ih_j, h_i) \qquad\text{(by ($5$))}\\
&= (g_iQh_j, h_i) \qquad\text{(by ($1$))}\\
&= -(Qh_j, h_i) \qquad\text{(by ($3$))}. \end{align*}
So $(Qh_j, h_i)=0$.
 \end{proof}
\begin{lemma}\label{C:LargePrConst} Let $P$ be a projection from $\ell_1(D_n)$ onto $Z(D_n)$. Then  $\|P\| \ge (2n+1)/3$. \end{lemma}
\begin{proof} By Theorem~\ref{thm: Z(D_n)}, we have  $\ell_1(D_n)=\operatorname{span}( \{h_0\} \cup (\cup_{k=0}^{2n-1} H_k))$ and
$Z(D_n)=\operatorname{span}(\cup_{k=0}^{n-1} H_{2k})$. By Proposition~\ref{prop: Andrew} it suffices to show that the orthogonal projection
$Q$ satisfies $\|Q\| \ge (2n+1)/3$.  This is well-known but for completeness we recall the proof. Consider
$$ f = h_0 + h_1 + 2h_2+ 2^2 h_4 +\dots + 2^{2n-2}h_{2^{2n-2}}.$$
Note that $f$ is the sum over the first Haar functions (normalized in $L_1[0,1]$) in each level. Then
$$Qf = h_1 + 2^2h_4 + 2^4h_{16} +\dots+ 2^{2n-2}h_{2^{2n-2}}.$$
It is easily seen that $\|f\| = 1$ and $\|Qf\| \ge (2n+1)/3$.
\end{proof} \begin{theorem}\label{C:LFDnNotL1} The Banach-Mazur distance $d$ from the Lipschitz free space $LF(D_n)$ to the $\ell_1^N$ space of the same dimension satisfies $$4n+4 \ge d \ge (2n+1)/3.$$ \end{theorem}

\begin{proof} Let $X_n = \operatorname{span}(\{h_0\} \cup (\cup_{k=0}^{2n-1} H_k))$ Using the inner product in $L_2[0,1]$ we may identify $\ell_1(D_n)^*$ with $(X_n, \|\cdot\|_\infty)$. Under this identification
$Z(D_n)^\perp = \operatorname{span}(\{h_0\}\cup(\cup_{k=1}^{n} H_{2k-1}))$. A calculation similar to that of the previous result,
but now using the $L_\infty$ norm, shows   that any projection $P$ from $(X_n, \|\cdot\|_\infty)$ onto $Z(D_n)^\perp$ satisfies $\|P\| \ge (2n+1)/3$. Since an $\ell_\infty^N$ space is contractively complemented in any superspace, it follows that the Banach-Mazur distance from  $LF(D_n)^*=Z(D_n)^\perp$ to an $\ell_\infty^N$ space
is at least $(2n+1)/3$. Dualizing again gives $d \ge (2n+1)/3$.

To get the upper estimate, note that  $\{h_0\}\cup(\cup_{k=1}^{n} H_{2k-1})$ is a monotone basis for $LF(D_n)$  in the quotient norm of $LF(D_n)$ and that
 $\{2^{2k-1}h_i \colon h_i \in H_{2k-1}\}$
is $2$-equivalent to the unit vector basis of the $\ell_1^N$ space of the same dimension.  Let $x \in LF(D_n)$ and write
$ x = \sum_{k=0}^n x_k$, where $x_0 \in \operatorname{span}(\{h_0\})$ and $x_k \in \operatorname{span}(H_{2k-1})$. Then
$$ \sum_{k=0}^n \|x_k\| \ge \|x\| \ge \frac{1}{2} \max_{0 \le k \le n} \|x_k\| \ge \frac{1}{2n+2} \sum_{k=0}^n \|x_k\|,$$
which gives $d \le 4n+4$.
\end{proof}

\begin{problem} Do $\{LF(D_n)\}$ admit embeddings into $\ell_1$ with uniformly
bounded distortions?\end{problem}

\begin{problem} Do $\{\ell_\infty^k\}$ admit embeddings into
$\{LF(D_n)\}$ with uniformly bounded distortions?\end{problem}

\begin{problem} Are $\{Z(D_n)\}_{n=1}^\infty$ uniformly isomorphic to $\{\ell_1^{k(n)}\}_{n=1}^\infty$ of the corresponding dimensions?  This is a finite version of the longstanding  open question as to
 whether  the even levels of the Haar system in $L_1[0,1]$ span a subspace isomorphic to $L_1$  \cite{MOP05}. \end{problem}

\begin{remark} It is curious that the subspaces generated by all the even/odd levels of the Haar functions appear in the study of quasi-greedy basic sequences in $L_1[0,1]$. The notion of quasi-greedy bases, which generalizes unconditional bases, was introduced by S. Konyagin and V. Temlyakov \cite{KT}, see also \cite{DKKT}]. Although the Haar basis is not quasi-greedy in $L_1[0,1]$ \cite{DKW02}, S. Gogyan \cite{Gog05} showed the subsequence consisting of all Haar functions from the even/odd levels is a quasi-greedy subsequence in $L_1[0,1]$. \end{remark}

Finally, we generalize the argument  to handle  the multi-branching diamond graphs $D_{n,k}$. The proof is similar to the case $k=2$, so we omit some of the  details.

\begin{theorem}\label{T:LFmultL1} The Banach-Mazur distance $d_{n,k}$ from the Lipschitz free space $LF(D_{n,k})$ to the $\ell_1^N$ space of the same dimension satisfies $$4n+4 \ge d_{n,k} \ge \frac{k-1}{2k}n.$$ \end{theorem}
\begin{proof}
It will be convenient to identify the edge space of $D_{n,k}$ with
a subspace of $L_1[0,1]$ as follows.  For $n=1$ and $1 \le j \le
k$ we identify the pair of  edge vectors of the $j^{th}$ path of
length $2$ from $u$ to $v$ with the $L_1$-normalized indicator
functions $2k \1_{(j-1)/k, (2j-1)/(2k)]}$ and $2k
\1_{((2j-1)/(2k), j/k]}$. For $n \ge2$,  the edge space of
$D_{n,k}$ is obtained from that of $D_{n,k-1}$ by subdividing the
intervals corresponding to edge vectors of $D_{n,k-1}$ into $2k$
subintervals each of length $(2k)^{-n}$.  Each of the $k$
consecutive disjoint  pairs of $L_1$-normalized indicator
functions of the subintervals corresponds to each pair of  edge
vectors of the $k$ paths of length $2$ from the $u$ and $v$
vertices of the copy of $D_{1,k}$ which  replaces the edge vector
of $D_{n-1,k}$   corresponding to the interval of length
$(2k)^{n-1}$ which is  subdivided. We have now identified
 the  edge vectors of $D_{n,k}$ with the normalized indicator functions
$$ e_{n,j} = (2k)^n \1_{((j-1)/(2k)^n, j/(2k)^n]} \quad (1 \le j \le (2k)^n).$$

Arguing as  in the case $k=2$, one can show that a   basis for the cycle space corresponds to the $L_\infty$-normalized system $\cup_{i=1}^n \{ g_{i,j}  \colon 1 \le j \le
(2k )^{i-1}(k-1)\}$, where, setting $j = a(k-1)+b$ with $0 \le a \le (2k)^{i-1}-1$ and $1 \le b \le k-1$,
$$g_{i,j} =(2k)^{-i}( e_{a2^k+2b-1} + e_{a 2^k + 2b} - e_{a 2^k + 2b + 1}- e_{a2^k+2b}).$$
Note that for $k=2$ this agrees with the previous description of the cycle space of $D_{n,2}$ in terms of alternate levels of the Haar system. But for $k \ge 3$, note that $g_{i,j}$ \textit{ overlaps}  with $g_{i,j+1}$  when $b \le k-2$, and hence this is not an orthogonal basis.

Using the fact that the cut space is the orthogonal complement of
the cycle space, it is easily seen that an orthogonal basis for
the  cut space   corresponds to the $L_\infty$-normalized system
$\{h_0\} \cup (\cup_{i=1}^n \{ h_{i,j} \colon 1 \le j \le
(2k)^i/2\}$, where $h_0 = \1_{[0,1]}$, and
$$ h_{i,j} 
= (2k)^{-i}(e_{i,2j-1}-e_{i,2j}).$$

Let $P_{n,k}$ denote the orthogonal projection from the edge space of $D_{n,k}$ onto the cut space.
Then $$P_{n,k}(e_{n,1}) =h_0 + \frac{1}{2} \sum_{i=1}^n (2k)^i h_{i,1}.$$
Note that for $1 \le i \le n$,
$$P_{n,k}(e_{n,1})|_{(2(2k)^{-i-1},(2k)^{-i}] }=1+\frac{1}{2}\sum_{j=1}^i(2k)^j
\ge\frac{(2k)^i}{2}.$$
Hence
$$\|P_{n,k}\|_1\ge \|P_{n,k}(e_{n,1})\|_1 \ge\sum_{i=1}^n (1-\frac{1}{k})(2k)^{-i} \frac{(2k)^i}{2} = (1-\frac{1}{k}) \frac{n}{2}. $$
Since $P_{n,k}$ is self-adjoint, it follows that $P_{n.k}$ is a
projection from the edge space $E(D_{n,k})$, equipped with the
$L_\infty$ norm, onto the cut space  $Z(D_{n,k})^\perp $
satisfying $\|P_{n,k}\|_\infty \ge (1- 1/k)n/2$.

 As in the case $k=2$, one can show that if $P$ is any projection onto the cut space (in the $L_\infty$ norm)  then $\|P\|_\infty \ge \|P_{n,k}\|_\infty$ .
By duality, as in the case $k=2$,  it follows that $d_{n,k} \ge (1 - 1/k)n/2$.

To get the upper estimate, note that   $\{h_0\} \cup (\cup_{i=1}^n \{ h_{i,j} \colon 1 \le j \le (2k)^i/2\})$ is a monotone basis
for $LF(D_{n,k})$  in the quotient norm of $LF(D_{n,k})$ and that,
 for each $i$, $(h_{i,j})_{j=1}^{(2k)^i}$
is $2$-equivalent to the unit vector basis of the $\ell_1^N$ space of the same dimension. As in the case $k=2$, this gives $d_{n,k} \le 4n+4$.
\end{proof}

\thanks{ \textbf{Acknowledgements:}
The first named author was supported by the National Science
Foundation under Grant Number DMS--1361461. The first and second
named authors were supported by the Workshop in Analysis and
Probability at Texas A\&M University in 2017. The third named
author was supported by the National Science Foundation under
Grant Number DMS--1700176.}

\end{large}
\begin{small}

\renewcommand{\refname}{
\end{small}

\textsc{Department of Mathematics, University of South Carolina, Columbia, SC 29208, USA.} \par \textit{E-mail address}: \texttt{dilworth@math.sc.edu} \par
\smallskip

\textsc{
 Department of Mathematics University of Illinois at Urbana-Champaign
 Urbana, IL 61801, USA and
Institute of Mathematics and Informatics,  Bulgarian Academy of Sciences,
  Sofia, Bulgaria.} \par
 \textit{E-mail address}: \texttt{denka@math.uiuc.edu}\par
\smallskip

\textsc{Department of Mathematics and Computer Science, St. John's
University, 8000 Utopia Parkway, Queens, NY 11439, USA} \par
  \textit{E-mail address}: \texttt{ostrovsm@stjohns.edu} \par


\section{References}}

\begin{thebibliography}{99}


\bibitem{ADIW09} A.~Andoni, K.~Do Ba, P.~Indyk, D.~Woodruff,
Efficient sketches for earth-mover distance, with applications.
{\it 2009 50th Annual IEEE Symposium on Foundations of Computer
Science FOCS 2009}, 324--330, IEEE Computer Soc., Los Alamitos,
CA, 2009.

\bibitem{AIK08} A.~Andoni, P.~Indyk, R.~Krauthgamer, Earth mover
distance over high-dimensional spaces. {\it Proceedings of the
Nineteenth Annual ACM-SIAM Symposium on Discrete Algorithms},
343--352, ACM, New York, 2008.

\bibitem{ANN17+}  A.~Andoni, A.~Naor, O.~Neiman, Snowflake universality of Wasserstein spaces, {\it Ann. Sci. \'Ec. Norm. Sup\'er.}, to
appear; {\tt arXiv:1509.08677}.

\bibitem{And78} A.\,D.~Andrew, On subsequences of the Haar system in $C(\Delta)$,
{\it Israel J. Math.} {\bf 31} (1978), 85--90.

\bibitem{AE56} R.\,F.~Arens, J.~Eells, Jr., On embedding uniform and topological spaces, {\it Pacific J. Math.},
{\bf  6}  (1956), 397--403.

\bibitem{BL00} Y.~Benyamini, J.~Lindenstrauss, {\it Geometric nonlinear
functional analysis}. Vol. {\bf 1}. American Mathematical Society
Colloquium Publications, {\bf 48}. American Mathematical Society,
Providence, RI, 2000.

\bibitem{Big97} N.\,L.~Biggs, Algebraic potential theory on graphs, {\it Bull.
London Math. Soc.}, {\bf 29} (1997),  no. 6, 641--682.

\bibitem{BM08} J.\,A.~Bondy, U.\,S.\,R.~Murty, {\it Graph theory}. Graduate Texts in
Mathematics, {\bf 244}. Springer, New York, 2008.

\bibitem{Bou85} J.~Bourgain,
On Lipschitz embedding of finite metric spaces in Hilbert space,
{\it Israel J. Math.}, {\bf 52} (1985), no. 1-2, 46--52.

\bibitem{Bou86} J.~Bourgain, The metrical interpretation
of superreflexivity in Banach spaces, {\it Israel J. Math.}, {\bf
56} (1986), no. 2, 222--230.

\bibitem{BS88} J.~Bourgain, S.\,J.~Szarek, The Banach-Mazur distance to the
cube and the Dvoretzky-Rogers factorization. {\it Israel J. Math.}
{\bf 62} (1988), no. 2, 169--180.

\bibitem{CM10} G.~Carlsson and F. M\'{e}moli, Characterization, Stability and Convergence of Hierarchical Clustering Methods,
{\it Journal of Machine Learning Research} \textbf{11} (2010)
1425--1470.

\bibitem{CD16} M.~C\'uth, M.~Doucha, Lipschitz-free spaces over ultrametric spaces. {\it Mediterr. J. Math.} {\bf 13} (2016), no. 4,
1893--1906.

\bibitem{CDW16} M.~C\'uth, M.~Doucha, P.~Wojtaszczyk, On the
structure of Lipschitz-free spaces. {\it Proc. Amer. Math. Soc.}
{\bf 144} (2016), no. 9, 3833--3846.

\bibitem{Dal15} A.~Dalet, Free spaces over some proper metric spaces.
{\it Mediterr. J. Math.} {\bf 12} (2015), no. 3, 973--986.

\bibitem{Die17} R.~Diestel, {\it Graph theory},
Fifth edition, Graduate Texts in Mathematics, {\bf 173},
Springer-Verlag, Berlin, 2017.

\bibitem{DKK} S.\,J.~Dilworth, N.J. Kalton, D.~Kutzarova,  On the existence of almost greedy bases in Banach spaces, {\it Studia Math.} {\bf 159} (2003), 67--101.

\bibitem{DKKT} S.J. Dilworth, N.J. Kalton, D. Kutzarova, V.\,N. Temlyakov,  The thresholding greedy algorithm, greedy basis, and duality, {\it Constr. Approx.}
{\bf 19} (2003), no. 4, 575--597.

\bibitem{DKW02}  S.J. Dilworth,  D. Kutzarova, P. Wojtaszczyk, On approximate $\ell_1$ systems in Banach spaces, {\it J. Approx. Theory} {\bf 114} (2002), 214--241.

\bibitem{Dob70} R.\,L.~Dobrushin, Definition of a system of random
variables by means of conditional distributions. {\it Teor.
Veroyatnost. i Primenen.} {\bf 15} (1970), 469--497; English
translation: {\it Theor. Probability Appl.} {\bf 15} (1970),
458--486.

\bibitem{DSW17} I.~Doust, S.~S\'anchez, A.~Weston, Asymptotic
negative type properties of finite ultrametric spaces. {\it J.
Math. Anal. Appl.} {\bf 446} (2017), no. 2, 1776--1793.


\bibitem{EP62} P.~Erd\H{o}s, L.~P\'osa, On the maximal number of disjoint circuits of
a graph. {\it Publ. Math. Debrecen} {\bf 9} (1962) 3--12.

\bibitem{Gia95} A.\,A.~Giannopoulos, A note on the Banach--Mazur distance
to the cube. Geometric aspects of functional analysis (Israel,
1992--1994), 67--73, {\it Oper. Theory Adv. Appl.}, {\bf 77},
Birkh\"auser, Basel, 1995.

\bibitem{God10}
A.~Godard,  Tree metrics and their Lipschitz-free spaces, {\it
Proc. Amer. Math. Soc.}, {\bf 138} (2010), no. 12, 4311--4320.

\bibitem{God15} G.~Godefroy, A survey on Lipschitz-free Banach spaces.
{\it Comment. Math.} {\bf 55} (2015), no. 2, 89--118.

\bibitem{GK03} G.~Godefroy, N.\,J.~Kalton, Lipschitz-free Banach spaces, {\it Studia Math.},
{\bf 159}  (2003),  no. 1, 121--141.

\bibitem{Gog05} S.~Gogyan, Greedy algorithm with regard to Haar subsystems,  {\it East J. Approx.} {\bf 11} (2005), 221--236.

\bibitem{Gru60} B.~Gr\"unbaum, Projection constants. {\it Trans. Amer. Math. Soc.} {\bf 95} (1960) 451--465.

\bibitem{Gup01} A.~Gupta, Steiner points in tree metrics don't (really)
help. {\it Proceedings of the Twelfth Annual ACM-SIAM Symposium on
Discrete Algorithms} (Washington, DC, 2001), 220--227, SIAM,
Philadelphia, PA, 2001.

\bibitem{GNRS04} A.~Gupta, I.~Newman, Y.~Rabinovich,
A.~Sinclair, Cuts, trees and $\ell_1$-embeddings of graphs, {\it
Combinatorica}, {\bf 24} (2004) 233--269; Conference version in:
{\it 40th Annual IEEE Symposium on Foundations of Computer
Science}, 1999, pp.~399--408.

\bibitem{IM04} P.~Indyk, J.~Matou\v sek, Low-distortion embeddings of finite metric spaces, in:
{\it Handbook of discrete and computational geometry}, ed. by
J.E.~Goodman and J.~O'Rourke, Boca Raton, Chapman and Hall/CRC,
2004, pp.~177--196.

\bibitem{JS09} W.\,B.~Johnson, G. Schechtman,
Diamond graphs and super-reflexivity, {\it J. Topol. Anal.}, {\bf
1} (2009), no. 2, 177--189.

\bibitem{KS71} M.\,I.~Kadets, M.\,G.~Snobar, Certain functionals on the
Minkowski compactum. (Russian) {\it Mat. Zametki} {\bf 10} (1971),
453--457.

\bibitem{Kal08} N.\,J.~Kalton, The nonlinear geometry of Banach
spaces, {\it Rev. Mat. Complut.}, {\bf 21} (2008), no. 1, 7--60.

\bibitem{Kan42} L.\,V.~Kantorovich, On mass transportation
(Russian), {\it Doklady Acad. Naus SSSR}, (N.S.) {\bf 37}, (1942),
199--201; English transl.: {\it J. Math. Sci.} (N. Y.), {\bf  133}
(2006), no. 4, 1381--1382.

\bibitem{KR58} L.\,V.~Kantorovich, G.\,S.~Rubinstein,
On a space of completely additive functions (Russian), {\it
Vestnik Leningrad. Univ.}, {\bf  13}  (1958), no. 7, 52--59.

\bibitem{KN06} S. Khot, A.~Naor, Nonembeddability
theorems via Fourier analysis, {\it Math. Ann.}, {\bf 334} (2006),
821--852.

\bibitem{KT} S.\,V.~Konyagin, V.\,N.~Temlyakov,  A remark on greedy
approximation in Banach spaces,{\it  East. J. Approx.} {\bf 5} (1999),
365--379.

\bibitem{Kru56} J.\,B.~Kruskal, Jr., On the shortest spanning subtree of a
graph and the traveling salesman problem. {\it Proc. Amer. Math.
Soc.} {\bf  7} (1956), 48--50.

\bibitem{Laa00} T.\,J.~Laakso, Ahlfors $Q$-regular spaces with arbitrary $Q >
1$ admitting weak Poincare inequality, {\it Geom. Funct. Anal.},
{\bf 10} (2000), no. 1, 111--123.

\bibitem{LP01} U.~Lang, C.~Plaut, Bilipschitz embeddings of metric
spaces into space forms, {\it Geom. Dedicata}, {\bf 87}  (2001),
285--307.

\bibitem{LR10} J.\,R.~Lee, P.~Raghavendra, Coarse differentiation and
multi-flows in planar graphs. {\it Discrete Comput. Geom.} {\bf
43} (2010), no. 2, 346--362.

\bibitem{LT77} J.~Lindenstrauss, L.~Tzafriri,
{\it Classical Banach spaces. {\bf I}. Sequence spaces}.
Ergebnisse der Mathematik und ihrer Grenzgebiete, Vol. {\bf 92}.
Springer-Verlag, Berlin-New York, 1977.

\bibitem{LLR95} N.~Linial, E.~London,
Y.~Rabinovich, The geometry of graphs and some of its algorithmic
applications, {\it Combinatorica}, {\bf 15} (1995), no. 2,
215--245.

\bibitem{MOP05} A. Mart\'inez-Abej\'on, E. Odell, M. M. Popov, Some open problems on the classical function space $L_1$,
Mat. Stud {\bf 24} (2005), 173--191.

\bibitem{NR17} A.~Naor, Y.~Rabani, On Lipschitz extension from finite
subsets. {\it Israel J. Math.} {\bf 219} (2017), no. 1, 115--161.

\bibitem{NS07} A. Naor, G. Schechtman, Planar Earthmover is not in
$L_1$, {\it SIAM J. Computing}, {\bf 37} (2007), 804--826.

\bibitem{OO17} S.~Ostrovska, M.\,I.~Ostrovskii,
Nonexistence of embeddings with uniformly bounded distortions of
Laakso graphs into diamond graphs, {\it Discrete Math.},  {\bf
340} (2017), no. 2, 9--17.

\bibitem{Ost13} M.\,I.~Ostrovskii, {\it Metric Embeddings: Bilipschitz and Coarse Embeddings into Banach Spaces},
de Gruyter Studies in Mathematics, {\bf 49}. Walter de Gruyter \&\
Co., Berlin, 2013.

\bibitem{OR17} M.\,I.~Ostrovskii, B.~Randrianantoanina, A new approach
to low-distortion embeddings of finite metric spaces into
non-superreflexive Banach spaces. {\it J. Funct. Anal.} {\bf 273}
(2017), no. 2, 598--651.

\bibitem{Pel00} D.~Peleg, {\it
Distributed computing. A locality-sensitive approach.} SIAM
Monographs on Discrete Mathematics and Applications, {\bf 5}.
Society for Industrial and Applied Mathematics (SIAM),
Philadelphia, PA, 2000.

\bibitem{Rud62} W.~Rudin, Projections on invariant subspaces. {\it Proc. Amer. Math. Soc.} {\bf 13} (1962) 429--432.

\bibitem{Sza90} S.\,J.~Szarek, Spaces with large distance to $\ell^n_\infty$ and
random matrices. {\it Amer. J. Math.} {\bf 112} (1990), no. 6,
899--942.

\bibitem{ST89} S.\,J.~Szarek, M.~Talagrand,  An ``isomorphic'' version of the
Sauer-Shelah lemma and the Banach-Mazur distance to the cube.
Geometric aspects of functional analysis (1987--88), 105--112,
{\it Lecture Notes in Math.}, {\bf 1376}, Springer, Berlin, 1989.

\bibitem{Tik18} K.~Tikhomirov, On the Banach-Mazur distance to cross-polytope, {\tt
arXiv:1804.08212}.

\bibitem{Vas69} L.\,N.~Vasershtein, Markov processes over denumerable products
of spaces describing large system of automata. {\it Problems of
Information Transmission} {\bf 5} (1969), no. 3, 47--52;
translated from: {\it Problemy Peredachi Informatsii} {\bf 5}
(1969), no. 3, 64--72.

\bibitem{Vil03} C.~Villani, {\it Topics in optimal transportation}. Graduate Studies in Mathematics, {\bf 58}. American Mathematical Society, Providence, RI, 2003.

\bibitem{Vil09} C.~Villani, {\it Optimal transport. Old and new.} Grundlehren der
Mathematischen Wissenschaften, {\bf 338}. Springer-Verlag, Berlin,
2009.

\bibitem{Wea99} N.~Weaver, {\it Lipschitz algebras}, Second edition, World Scientific
Publishing Co. Pte. Ltd., Hackensack, NJ, 2018.

\bibitem{You14} P.~Youssef,  Restricted invertibility and the
Banach-Mazur distance to the cube. {\it Mathematika} {\bf 60}
(2014), no. 1, 201--218.

\end{thebibliography}
\end{document}